\documentclass[reqno,a4paper, 11pt]{amsart}

\usepackage[a4paper=true,pdfpagelabels]{hyperref}
\usepackage{graphicx}

\usepackage[ansinew]{inputenc}
\usepackage{amsfonts,epsfig}
\usepackage{latexsym}
\usepackage{amsmath}
\usepackage{amssymb}

\newtheorem{theorem}{Theorem}
\newtheorem{lemma}[theorem]{Lemma}
\newtheorem{corollary}[theorem]{Corollary}
\newtheorem{proposition}[theorem]{Proposition}

\newtheorem{lettertheorem}{Theorem}
\newtheorem{letterlemma}[lettertheorem]{Lemma}

\theoremstyle{definition}

\theoremstyle{remark}

\numberwithin{equation}{section}

\newcommand{\B}{\mathcal{B}}
\newcommand{\D}{\mathbb{D}}
\newcommand{\DD}{\widehat{\mathcal{D}}}
\newcommand{\N}{\mathbb{N}}
\newcommand{\R}{\mathbb{R}}

\newcommand{\C}{\mathbb{C}}

\newcommand{\e}{\varepsilon}

\renewcommand{\phi}{\varphi}

\newcommand{\T}{\mathbb{T}}

\def\a{\alpha}       \def\b{\beta}        
\def\d{\delta}           \def\e{\varepsilon}
     \def\om{\omega}      
              
                  \def\z{\zeta}

\def\R{{\mathcal R}}
\def\I{{\mathcal I}}

\DeclareMathOperator{\supp}{supp}
\renewcommand{\H}{\mathcal{H}}

\begin{document}

\title[Two weight inequality for Bergman projection]{Two weight inequality for Bergman projection}


\thanks{This research was supported in part by the Ram\'on y Cajal program
of MICINN (Spain); by Ministerio de Edu\-ca\-ci\'on y Ciencia, Spain, projects
MTM2011-25502 and MTM2011-26538;  by   La Junta de Andaluc{\'i}a, (FQM210) and
(P09-FQM-4468);  by Academy of Finland project no. 268009,  by V\"ais\"al\"a Foundation of Finnish Academy of Science and Letters, and by Faculty of Science and Forestry of University of Eastern Finland project no. 930349.
}

\date{\today}

\begin{abstract}
The motivation of this paper comes from the two weight inequality $$\|P_\om(f)\|_{L^p_v}\le C\|f\|_{L^p_v},\quad f\in L^p_v,$$ for the Bergman projection $P_\om$ in the unit disc. We show that the boundedness of $P_\om$ on $L^p_v$ is characterized in terms of self-improving Muckenhoupt and Bekoll\'e-Bonami type conditions when the radial weights $v$ and $\om$ admit certain smoothness. En route to the proof we describe the asymptotic behavior of the $L^p$-means and the $L^p_v$-integrability of the reproducing kernels of the weighted Bergman space $A^2_\om$.
\end{abstract}

\author{Jos\'e \'Angel Pel\'aez}
\address{Departamento de An\'alisis Matem\'atico, Universidad de M\'alaga, Campus de
Teatinos, 29071 M\'alaga, Spain} \email{japelaez@uma.es}

\author{Jouni R\"atty\"a}
\address{University of Eastern Finland, P.O.Box 111, 80101 Joensuu, Finland}
\email{jouni.rattya@uef.fi}

\maketitle

\section{Introduction}

Let $A^2_\om$ denote the subspace of analytic functions in $L^2_\om$ induced by a nonnegative integrable function $\om$ on the unit disc $\D$. If the norm convergence in the Bergman space $A^2_\om$ implies the uniform convergence on compact subsets, the Hilbert space $A^2_\om$ is a closed subspace of $L^2_\om$ and the orthogonal Bergman projection $P_\om$ from $L^2_\om$ to $A^2_\om$ is given by
    \begin{equation*}\label{intoper}
    P_\om(f)(z)=\int_{\D}f(\z)\overline{B^\om_{z}(\z)}\,\om(\z)dA(\z),
    \end{equation*}
where $B^\om_{z}$ are the reproducing kernels of $A^2_\om$.

In this paper we are mainly interested in the question of when
    \begin{equation}\label{twoweight}
    \|P_\om(f)\|_{L^p_v}\le C\|f\|_{L^p_v}.
    \end{equation}
To the best of our knowledge, the existing literature does not  offer an answer even in the case where $\om=v$ is radial. It is well-known that the boundedness of projections on $L^p$-spaces is an intriguing topic which has attracted a considerable amount of attention during the last decades. This is not only due to the mathematical difficulties the question raises, but also to its numerous applications in operator theory.
Recently, the bounded projections $P_0:L^2_{|g|^{-2}}\to L^2_{|f|^2}$ were characterized on the way to disprove the Sarason conjecture on the Toeplitz product operator $T_fT_g^\star:A^2\to A^2$, induced by analytic symbols $f$ and $g$~\cite{AlPoRe}. However,
the most commonly known results concerning the two weight inequality \eqref{twoweight} have been obtained when
the inducing weight $\om$ is standard~\cite{BB,B}. In this case the reproducing kernels are given by the neat expression $(1-\overline{z}\z)^{-(2+\alpha)}$ that is easy to work with.
The general situation is much more complicated because of the lack of explicit expressions for $B^\om_z$. Because of this fact, and due to previous studies \cite{CP2,D1,ZeyTams2012} revealing the importance that the decay of the weight plays in the analysis of~\eqref{twoweight}, we will focus on so-called regular and rapidly increasing (radial) weights. Postponing the exact definitions of these weights to the next section, we will denote these classes of weights by~$\R$ (for regular) and~$\I$ (for rapidly increasing).

The techniques employed here to study~\eqref{twoweight} require $L^p$-estimates for the Bergman reproducing kernels~$B^\om_{z}$. The first of the main results describes the asymptotic behavior of the $L^p$-means
of $B^\om_{z}$ (or its derivatives). The latter part of this theorem reveals a precise estimate for the $L^p_v$-integral of~$B^\om_{z}$.
Needless to say that such kernel estimates are frequently applied in the operator theory.

The main result of this study characterizes those regular weights $\om$ and $v$ for which \eqref{twoweight} holds. In particular we show that they coincide with those for which the sublinear operator
    $$
    P^+_\om(f)(z)=\int_{\D}|f(\z)||B^\om_{\z}(z)|\,\om(\z)dA(\z)
    $$
is bounded on $L^p_v$. The characterizing integral condition is equivalent, on one hand, to a Muckenhoupt-type condition related to Hardy operators~\cite{Muckenhoupt1972}, and on the other hand, to a generalization of the classical Bekoll\'e-Bonami condition. In contrast to the general situation for Bekoll\'e-Bonami weights~\cite{BoMathAnn04}, all these conditions are self-improving.

\par As a byproduct, we will show that $P^+_\om$ is bounded on $L^p_\om$ if $\om\in\R$ and $p>1$. The situation is different for $\om\in\I$ because then $P^+_\om$ is not bounded on $L^p_\om$. These results emphasize the general phenomena that many finer function-theoretic properties valid for $A^p_\alpha$ just simply break down for $A^p_\om$ induced by $\om\in\I$.

Throughout the paper  $\frac{1}{p}+\frac{1}{p'}=1$. Further, the letter $C=C(\cdot)$ will denote an
absolute constant whose value depends on the parameters indicated
in the parenthesis, and may change from one occurrence to another.
We will use the notation $a\lesssim b$ if there exists a constant
$C=C(\cdot)>0$ such that $a\le Cb$, and $a\gtrsim b$ is understood
in an analogous manner. In particular, if $a\lesssim b$ and
$a\gtrsim b$, then we will write $a\asymp b$.

\subsection{Background on weights}

Before presenting the main results, we will shortly discuss some classes of radial weights.

A function $\omega:\D\to [0,\infty)$, integrable over the unit disc $\D$, is called a
weight. It is radial if $\omega(z)=\omega(|z|)$ for all $z\in\D$.
We will write $\DD$ for the class of radial weights such that $\widehat{\om}(z)=\int_{|z|}^1\om(s)\,ds$ is doubling, that is, there exists $C=C(\om)\ge1$ such that $\widehat{\om}(r)\le C\widehat{\om}(\frac{1+r}{2})$ for all $0\le r<1$. The following lemma contains basic properties of these weights and will be frequently used in the sequel. The proof is elementary and therefore omitted.

\begin{letterlemma}\label{Lemma:replacement-Lemmas-Memoirs}
Let $\om$ be a radial weight. Then the following conditions are equivalent:
\begin{itemize}
\item[\rm(i)] $\om\in\DD$;
\item[\rm(ii)] There exist $C=C(\om)>0$ and $\b_0=\b_0(\om)>0$ such that
    \begin{equation}\label{Eq:replacement-Lemma1.1}
    \begin{split}
    \widehat{\om}(r)\le C\left(\frac{1-r}{1-t}\right)^{\b}\widehat{\om}(t),\quad 0\le r\le t<1,
    \end{split}
    \end{equation}
for all $\b\ge\b_0$;

\item[\rm(iii)] The asymptotic equality
    $$
    \int_0^1s^x\om(s)\,ds\asymp\widehat{\om}\left(1-\frac1x\right),\quad x\in[1,\infty),
    $$
is valid;

\item[\rm(iv)] $\om^\star(z)\asymp\widehat{\om}(z)(1-|z|)$, $|z|\to1^-$,
where
    $$
    \omega^\star(z)=\int_{|z|}^1\omega(s)\log\frac{s}{|z|}s\,ds,\quad z\in\D\setminus\{0\}.
    $$
\end{itemize}
\end{letterlemma}

Each radial weight $\om$ is closely related to its associated weight $\omega^\star$ by the Littlewood-Paley identity
    \begin{equation}\label{LP1}
    \|f\|^2_{A^2_\om}=4\|f'\|^2_{A^2_{\om^\star}}+\om(\D)|f(0)|^2,
    \end{equation}
which is a special case of a more general formula~\cite[Theorem~4.2]{PelRat}.

We call a radial weight $\om$ regular, denoted by $\om\in\R$, if $\om\in\DD$
and  $\om(r)$ behaves as its integral average over $(r,1)$, that is,
    \begin{equation*}
    \om(r)\asymp\frac{\int_r^1\om(s)\,ds}{1-r},\quad 0\le r<1.
    \end{equation*}
It is clear that  $\om\in\R$ if and only if for each $s\in[0,1)$ there exists a
constant $C=C(s,\omega)>1$ such that
    \begin{equation}\label{eq:r2}
    C^{-1}\om(t)\le \om(r)\le C\om(t),\quad 0\le r\le t\le
    r+s(1-r)<1,
    \end{equation}
and
\begin{equation*}
   \frac{\int_r^1\om(s)\,ds}{1-r}\lesssim \om(r),\quad0\le r<1.
    \end{equation*}
The definition of regular weights used here is slightly more general than that in \cite{PelRat}, but the principal properties of weights in these classes are essentially the same by Lemma~\ref{Lemma:replacement-Lemmas-Memoirs} and \cite[Chapter~1]{PelRat}. A radial continuous weight $\om$ is called rapidly increasing, denoted by $\om\in\I$, if
    \begin{equation*}
    \lim_{r\to 1^-}\frac{\int_r^1\om(s)\,ds }{\om(r)(1-r)}=\infty.
    \end{equation*}
It follows from \cite[Lemma~1.1]{PelRat} that $\I\subset\DD$. For further information on these classes, see~\cite[Chapter~1]{PelRat} and the references therein.

\section{Main results}

Let $\H(\D)$ denote the algebra of all analytic functions in the unit disc $\D=\{z:|z|<1\}$.
If $0<r<1\,$ and $f\in \H (\D)$, set
    \begin{equation*}
    \begin{split}
    M_p(r,f)&=\left(\frac{1}{2\pi}\int_{0}^{2\pi} |f(re^{it})|^p\,dt\right)^{1/p},\quad
    0<p<\infty,\\
    M_\infty(r,f)&=\sup_{|z|=r}|f(z)|.
    \end{split}
    \end{equation*}
For $0<p\le \infty $, the Hardy space $H^p$ consists of functions
$f\in \H(\mathbb D)$ such that $\Vert f\Vert _{H^p}=
\sup_{0<r<1}M_p(r,f)<\infty$. For
$0<p<\infty$ and a weight $\omega$, the weighted Bergman
space $A^p_\omega$ is the space of  $f\in\H(\D)$ for
which
    $$
    \|f\|_{A^p_\omega}^p=\int_\D|f(z)|^p\omega(z)\,dA(z)<\infty,
    $$
where $dA(z)=\frac{dx\,dy}{\pi}$ is the normalized
Lebesgue area measure on $\D$. As usual, we write~$A^p_\alpha$ for the classical weighted
Bergman space induced by
the standard radial weight $\omega(z)=(1-|z|^2)^\alpha$, where
$-1<\alpha<\infty$.

If $\om\in\DD$, the norm convergence in $A^2_\om$ implies the uniform convergence on compact subsets of $\D$, and therefore $A^2_\om$ is a closed subspace of $L^2_\om$. In particular, each point evaluation $L_ a(f)=f(a)$ is a bounded linear
functional on $A^2_\om$, and hence there exist unique reproducing kernels $B^\omega_a\in
A^2_\om$ such that $\|L_ a\|=\|B^\omega_
a\|_{A^2_\om}$ and
    $$
    f(a)=\langle f, B^\omega_ a\rangle_{A^2_\om} =\int_{\D}
    f(z)\,\overline{B^\omega_a(z)}\,\om(z)\,dA(z),\quad f\in A^2_\om.
    $$

When a closed formula for the Bergman kernel $B^\om_a$ exists, then the asymptotic growth of its $L^p$-means can be determined. For example, if the inducing weight is $\om(z)=(1-|z|^2)^\a$, then an appropriate interpretation of the well-known $L^p$-estimate allows us to write
    \begin{equation*}
    \begin{split}
    M_p^p\left(r,B^\om_a\right)\asymp\int_0^{|a|r}\frac{dt}{(1-t)^{(2+\a)p}}
    \asymp\int_0^{|a|r}\frac{dt}{\widehat{\om}(t)^p(1-t)^{p}},\quad r,|a|\to1^-,
    \end{split}
    \end{equation*}
and, for $v(z)=(1-|z|^2)^{\b}$, we therefore have
    \begin{equation*}
    \begin{split}
    \|B^\om_a\|^p_{A^p_v}
    \asymp\int_0^1 (1-r)^\b\left(\int_0^{|a|r}\frac{dt}{(1-t)^{(2+\a)p}}\right)\,dr
    \asymp\int_0^{|a|}\frac{\widehat{v}(r)}{\widehat{\om}(r)^p(1-r)^p}\,dr,\quad |a|\to1^-.
    \end{split}
    \end{equation*}
This last one is a standard Bergman kernel estimate in the unit disc, attributed to Forelli and Rudin~\cite{ForRud74Ind}, that is usually written in a slightly different form~\cite{HKZ,Zhu}. The $L^p$-behavior of the kernel $B^\om_a$ can be controlled in terms of off-diagonal pointwise estimates if the inducing radial weight tends to zero at least exponentially as $|z|\to 1^-$ \cite{ Asserda-Hichame2014,ArPau,Scuster-Varolin2013}.

Our first result shows that the discussion above regarding standard weights actually describes a general phenomenon rather than a particular case.

\begin{theorem}\label{th:kernelstimate}
Let $0<p<\infty$, $\om\in\DD$ and $N\in\N\cup\{0\}$. Then the following assertions hold:
\begin{enumerate}
\item[\rm(i)]$\displaystyle M_p^p\left(r,\left(B^\om_a\right)^{(N)}\right)\asymp
    \int_0^{|a|r}\frac{dt}{\widehat{\om}(t)^{p}(1-t)^{p(N+1)}},\quad r,|a|\to1^-.$
\item[\rm(ii)] If $v\in\DD$, then
    \begin{equation}\label{k1}
    \|\left(B^\om_a\right)^{(N)}\|^p_{A^p_v}
    \asymp \int_0^{|a|}\frac{\widehat{v}(t)}{\widehat{\om}(t)^{p}(1-t)^{p(N+1)}}\,dt,\quad |a|\to1^-.
    \end{equation}
\end{enumerate}
\end{theorem}

It is clear by the proof that the asymptotic inequality $\lesssim$ in \eqref{k1} is actually valid for any radial weight $v$, see \eqref{anyradial} below. The following consequence of Theorem~\ref{th:kernelstimate} is often more useful than the theorem itself.

\begin{corollary}\label{co:ernelstimaten}
Let $0<p<\infty$, $\om\in\DD$  and $N\in\N\cup\{0\}$. Then the following assertions hold.
\begin{enumerate}
\item[\rm(i)]
    $\displaystyle
    M_p^p\left(r,\left(B^\om_a\right)^{(N)}\right)\asymp \frac{1}{\widehat{\om}(ar)^{p}(1-|a|r)^{p(N+1)-1}},\quad r, |a|\to1^-,
    $\\
if and only if
    \begin{equation}\label{Eq:hypothesis-kernelmeans}
    \int_0^{|a|}\frac{dt}{\widehat{\om}(t)^{p}(1-t)^{p(N+1)}}\lesssim\frac{1}{\widehat{\om}(a)^{p}(1-|a|)^{p(N+1)-1}},\quad |a|\to1^-.
    \end{equation}
\item[\rm(ii)] If $v\in\DD$, then
\begin{equation}\label{kn4}
    \|\left(B^\om_a\right)^{(N)}\|^p_{A^p_v}\asymp\frac{\widehat{v}(a)}{\widehat{\om}(a)^{p}(1-|a|)^{p(N+1)-1}},\quad |a|\to1^-,
    \end{equation}
    if and only if
    \begin{equation}\label{Eq:hypothesis-kernel}
    \int_0^r\frac{\widehat{v}(t)}{\widehat{\om}(t)^{p}(1-t)^{p(N+1)}}\,dt\lesssim\frac{\widehat{v}(r)}{\widehat{\om}(r)^{p}(1-r)^{p(N+1)-1}},\quad r\to1^-.
    \end{equation}
\end{enumerate}
\end{corollary}

There are two instances in the recent literature where the Bergman kernel $B^\om_a$ induced by a standard weight is estimated. In the first one $v$ is assumed to be related to the classical Bekoll\'e-Bonami weights~\cite[Lemma~2.1]{AlCo}, and in the second one $v\in\I\cup\R$~\cite[Lemma~2.3(a)]{PelRat}.

The proof of Theorem~\ref{th:kernelstimate} consists of several steps. First, we deduce the upper bound in (i) for $p>2$ (and the lower bound for $p<2$) by using Hardy-Littlewood inequalities. These estimates could then be used to establish the corresponding cases in (ii), but in order to give a more uniform treatment, we will argue differently. Indeed, as the second step, we will prove (ii) for $v\in\R$ continuous by using a Littlewood-Paley theorem (to boost the order of differentiation), decomposition norm estimates for $A^p_v$ with a precise control (induced by the regularity of $v$) over the size of the blocks, results on smooth polynomials related to Hadamard products, and the lower bound for $p\le1$ in (i). As the third step,
we will apply (ii) for $v(z)=(1-|z|)^{p-1}$ and classical embeddings to obtain~(i). The final step is to deduce (ii) for $v\in\DD$ from (i).

Once we get Theorem~\ref{th:kernelstimate} we turn to consider the Bergman projection $P_\om$ acting on an $L^p$-space that is induced by a different weight than the kernel itself. This leads us to study the two weight inequality \eqref{twoweight}.
Bekoll\'e and Bonami described the weights (not necessarily radial) such that the Bergman projection
    \begin{equation*}
    P_\a(f)(z)=(\a+1)\int_\D \frac{f(\z)(1-|\z|^2)^\a}{(1-z\overline{\z})^{2+\a}}\,dA(\z),\quad \a>-1,
    \end{equation*}
induced by the standard weight $\om(z)=(\a+1)(1-|z|^2)^\a$, is bounded on $L^p_v$ for $p>1$~\cite{BB,B}.
They also showed that these weights are exactly those for which the sublinear operator
    $$
    P^+_\a(f)(z)=(\a+1)\int_\D \frac{|f(\z)|(1-|\z|^2)^\a}{|1-z\overline{\z}|^{2+\a}}\,dA(\z)
    $$
is bounded on $L^p_v$. It is worth mentioning that even if the Bekoll\'e-Bonami weights are a kind of analogue of the Muckenhoupt class, these classes have significant differences~\cite{BoMathAnn04}.

The next theorem is the main result of this paper.

\begin{theorem}\label{theorem:projections2}
Let $1<p<\infty$ and $\om,v\in\R$. Then the following conditions are equivalent:
\begin{enumerate}
\item[\rm(a)] $P^+_\om:L^p_v\to L^p_v$ is bounded;
\item[\rm(b)] $P_\om:L^p_v\to L^p_v$ is bounded;
\item[\rm(c)]  $\left(\frac{\om}{v}\right)^{p'}\,v$ is a regular weight.
\end{enumerate}
\end{theorem}

To prove Theorem~\ref{theorem:projections2}, we will first use the boundedness of the adjoint of $P_\om$, with the monomials as test functions, to see that
    \begin{equation}\label{Eq:BB-intro}
   \sup_{0<r<1} \frac{\widehat{v}(r)^{\frac{1}{p}}
    \left(\int_r^1\left(\frac{\om(s)}{v(s)}\right)^{p'}\,v(s)ds\right)^{\frac{1}{p'}}
    }{\widehat\om(r)}<\infty,
     \end{equation}
that is, the integrand is a regular weight. If $\om(z)=(1-|z|^2)^\a$, then this is the same as saying that the radial weight $\frac{v}{\om}$ satisfies the corresponding Bekoll\'e-Bonami condition. The more involved part of the proof is to show that \eqref{Eq:BB-intro} is also a sufficient condition, and that will be achieved by using an instance of Schur's test and Theorem~\ref{th:kernelstimate}. Further, we will show that \eqref{Eq:BB-intro} is equivalent to the Muckenhoupt-type condition on Hardy operators
    \begin{equation}\label{Eq:Muckenhoupt-intro}
    \displaystyle \sup_{0<r<1}\left(\int_0^r\frac{v(s)}{\om(s)^p(1-s)^p}\,ds\right)^{\frac{1}{p}}
    \left(\int_r^1\left(\frac{\om(s)}{v(s)}\right)^{p'}\,v(s)ds\right)^{\frac{1}{p'}}<\infty,
    \end{equation}
which coincides with
    \begin{equation}\label{Eq:improving-intro}
   \sup_{0<r<1} \frac{\widehat{\om}(r)^p}{\widehat{v}(r)}\int_0^r  \frac{\widehat{v}(s)}{\widehat{\om}(s)^p (1-s)}\,ds<\infty.
    \end{equation}
The condition \eqref{Eq:Muckenhoupt-intro} follows also directly by the boundedness of $P^+_\om$. Therefore several equivalent integral conditions characterize the boundedness of $P_\om$ on $L^p_v$ when $1<p<\infty$. The condition \eqref{Eq:improving-intro}, as well as all the others, is self-improving in the sense that if it is satisfied for some $p$, then it is also satisfied when $p$ is replaced by $p-\d$, where $\d>0$ is sufficiently small, see Lemma~\ref{Lemma:self-improving} below. Recall that the Bekoll\'e-Bonami condition is not self-improving in general \cite{BoMathAnn04}.

It is worth noticing that \eqref{Eq:improving-intro} makes sense also for $p=1$, and it turns out to be the right condition for describing those regular weights such that $P_\om$ is bounded on $L^1_v$.

\begin{theorem}\label{theorem:projections5}
Let $\om,v\in\R$. Then the following conditions are equivalent:
\begin{itemize}
\item[\rm(a)] $P_\om:L^1_v\to L^1_v$ is bounded;
\item[\rm(b)] $P^+_\om:L^1_v\to L^1_v$ is bounded;
\item[\rm(c)] $\displaystyle \sup_{0<r<1}\frac{\om(r)}{v(r)}\int_0^r\frac{\widehat{v}(s)}{\widehat{\om}(s)(1-s)}\,ds<\infty$;
\item[\rm(d)]$\displaystyle   \sup_{0<r<1}\frac{\widehat{v}(r)}{\widehat{\om}(r)}\int_r^1\frac{\om(s)}{v(s)(1-s)}\,ds<\infty$.
\end{itemize}
\end{theorem}

We will also show that \eqref{twoweight} is equivalent to the inequality $\kappa_\om<p\kappa_v$, whenever $\om,v\in\R$ are
such that $\kappa_\om=\lim_{r\to1^-}\frac{\psi_\om(r)}{1-r}$ and $\kappa_v=\lim_{r\to1^-}\frac{\psi_v(r)}{1-r}$ exist and $1\le p<\infty$.
This neat inequality reduces to the known condition $(\b+1)<p(\a+1)$ when $\om(z)=(1-|z|)^\a$ and $v(z)=(1-|z|)^\beta$, see~\cite{ForRud74Ind} and \cite[Theorem~4.24]{Zhu}.

It immediately follows from Theorem~\ref{theorem:projections2} that $P^+_\om$ is bounded on $L^p_\om$ when $\om\in\R$ and $1<p<\infty$. We will see that this does not remain true if $\om\in\I$, and therefore cancelation plays a role when $A^p_\om$ is sufficiently close to the Hardy space~$H^p$. It is worth mentioning that $P_\om$ fails to be bounded on $L^p_\om$ if $\om$ decreases sufficiently fast (at least exponentially) and is smooth enough~\cite{CP2,D1,D3,ZeyTams2012}. In fact, as far as we know, to characterize those radial weights for which $P_\om:L^p_\om\to L^p_\om$ is bounded, is an open problem~\cite[p.~116]{D1}. 
Regarding the case $p=\infty$, we prove that $P_\om:L^\infty(\D)\to \mathcal{B}$ is bounded if $\om\in\R$. Here~$\mathcal{B}$ denotes the Bloch space that consists of $f\in\H(\D)$ such that
    $$
    \|f\|_{\mathcal{B}}=\sup_{z\in\D}|f'(z)|(1-|z|^2)+|f(0)|<\infty.
    $$
These results are gathered in the following theorem.

\begin{theorem}\label{theorem:projections}
Let $1<p<\infty$.
\begin{enumerate}
\item[\rm(i)] If $\om\in\R$, then $P^+_\om:L^p_\om\to L^p_\om$ is bounded. In particular, $P_\om:L^p_\om\to A^p_\om$ is bounded.
\item[\rm(ii)] If $\om\in\R$, then $P_\om:L^\infty(\D)\to \mathcal{B}$ is bounded.
\item[\rm(iii)] If $\om\in\I$, then $P^+_\om$ is not bounded from $L^p_\om$ to $L^p_\om$.
\end{enumerate}
\end{theorem}

The projection $P_\om$ is not bounded on $L^1_\om$ if $\om$ is continuous. However, for $\om\in\R$
there are plenty of bounded projections on $L^1_\om$, as Theorem~\ref{theorem:projections5} shows. See also \cite[Proposition~2.1]{AlCo} and \cite[Lemma~2.1]{PelRat}.
The situation is completely different for $\om\in\I$ by a result due to Shields and Williams~\cite[Theorem~3]{ShiWiMich82}. For the sake of
completeness, we will rewrite this result in our language to show that there are no bounded projections from $L^1_\om$ to $A^1_\om$ if $\om\in\I$ is smooth enough.

\begin{lettertheorem}\label{th:L1pI}
Let $\om\in\I$ and assume that there exists an increasing function $\Psi:[0,\infty)\to [0,\infty)$, convex or concave, such that
    $$
    \Psi(x)\asymp \frac{1}{\widehat{\om}\left(1-\frac{1}{x+1}\right)},\quad x\in[0,1).
    $$
Then, there are no bounded projections from $L^1_\om$ to $A^1_\om$.
\end{lettertheorem}

This result is strongly connected with the fact that there are no bounded projections from $L^1$ of the unit circle to
$H^1$~\cite[Theorem~9.7]{Zhu}.

The boundedness of projections plays an important role in many characterizations of dual spaces, and therefore it is natural to expect that
Theorem~\ref{theorem:projections2} can be used to establish such results. It turns out that \eqref{twoweight} can be reformulated in terms of a duality relation when the weights are regular.

\begin{theorem}\label{th:duality3}
Let $1<p<\infty$ and $\om,v\in\R$, and denote $V_{p'}=V_{p'}(\om,v)=\left(\frac{\om}{v}\right)^{p'}v$. Then the following assertions are equivalent:
\begin{itemize}
\item[\rm(a)] $P_\om:L^p_v\to L^p_v$ is bounded;
\item[\rm(b)] The dual of $A^{p'}_{V_{p'}}$ can be identified with $A^p_v$ (up to an equivalence of norms) under the pairing
    \begin{equation}\label{pairingom}
    \langle f,g\rangle_{A^2_\om}=\int_{\D}f(z)\overline{g(z)}\,\om(z)dA(z).
    \end{equation}
    \end{itemize}
\end{theorem}

The argument used in the proof readily shows that $(A^p_v)^\star\simeq A^{p'}_{V_{p'}}$, if $P_\om:L^p_v\to L^p_v$ is bounded, and hence, in this case, $A^p_v$ is reflexive.

Finally, we will discuss two cases which are probably the most neat ones in this context. Part~(i) of the next result follows from Theorem~\ref{th:duality3}, and Part~(ii) is probably known, at least to experts working on the field.

\begin{corollary}\label{th:duality2}
Let $1<p<\infty$ and $\om\in\R$. Then the following assertions hold under the pairing~\eqref{pairingom}:
\begin{itemize}
\item[\rm(i)] $(A^p_\om)^\star\simeq A^{p'}_\om$;
\item[\rm(ii)] $(A^1_\om)^\star\simeq\B$.
\end{itemize}
\end{corollary}

\section{Integrability of reproducing kernels}

In this section we will prove Theorem~\ref{th:kernelstimate} and then deduce Corollary~\ref{co:ernelstimaten}. We will need several auxiliary results that are presented first.

\subsection{Preliminary results}

We begin with auxiliary results on smooth Hadamard products, and then apply Hardy-Littlewood-inequalities to obtain estimates for $L^p$-means of the reproducing kernels.

Throughout this section we will assume, without loss of generality, that
$\int_0^1 \om(s)\,ds=1$. For each $n\in\N\cup\{0\}$, let $r_n=r_n(\om)\in[0,1)$
be defined by
    \begin{equation}\label{rn}
    \widehat{\om}(r_n)=\int_{r_n}^1 \om(s)\,ds=\frac{1}{2^n}.
    \end{equation}
Clearly, $\{r_n\}_{n=0}^\infty$ is a non-decreasing sequence of
distinct points on $[0,1)$ such that $r_0=0$ and $r_n\to1^-$, as
$n\to\infty$. For $x\in[0,\infty)$, let $E(x)$ denote the integer
such that $E(x)\le x<E(x)+1$, and set
$M_n=E\left(\frac{1}{1-r_{n}}\right)$. Write
    $$
    I(0)=I_{\om}(0)=\left\{k\in\N\cup\{0\}:k<M_1\right\}
    $$
and
   \begin{equation*}
    I(n)=I_{\om}(n)=\left\{k\in\N:M_n\le
    k<M_{n+1}\right\}
  \end{equation*}
for all $n\in\N$. If
$f(z)=\sum_{n=0}^\infty a_nz^n$ is analytic in~$\D,$ define the
polynomials $\Delta^{\om}_nf$ by
    \[
    \Delta_n^{\om}f(z)=\sum_{k\in I_{\om}(n)} a_kz^k,\quad n\in\N\cup\{0\}.
    \]
The next result on partial sums $\Delta_n^{\om}f$ together with \cite[Theorem~4]{PelRathg} is one of the
principal ingredients in the proof of Theorem~\ref{th:kernelstimate}.

\begin{lemma}\label{le:de1}
Let $0<p\le 1$ and $\om\in\DD$. Then
    \[ \|f\|^p_{A^p_\om}\lesssim
    \sum_{n=0}^\infty 2^{-n} \|\Delta^{\om}_n
    f\|_{H^p}^p\]
for all $f\in\H(\D)$.
\end{lemma}

\begin{proof}
Since $0<p\le 1$, \cite[(3.13)]{PelRathg} yields
    \begin{equation*}
    M^p_p(r,f)\le\sum_{n=0}^\infty M^p_p(r,\Delta^{\om}_nf)\le\sum_{n=0}^\infty
    r^{pM_n}\|\Delta^{\om}_n f\|^p_{H^p}.
    \end{equation*}
With Lemma~\ref{Lemma:replacement-Lemmas-Memoirs} in hand, one readily sees that \cite[Lemma~8(i) and (ii)]{PelRathg} is valid for $\om\in\DD$, and hence \cite[Proposition~9]{PelRathg} also. This latter result, with $p=1=\a$, gives
    \begin{equation*}
    \begin{split}
    \|f\|^p_{A^p_\om}\le\int_0^1 \left( \sum_{n=0}^\infty
    r^{pM_n}\|\Delta^{\om}_n f\|^p_{H^p}\right)\om(r)\,dr
    \asymp\sum_{n=0}^\infty 2^{-n}\|\Delta^{\om}_n
    f\|_{H^p}^p,
    \end{split}
    \end{equation*}
which is the inequality we wanted to prove. The last asymptotic equality can also be directly deduced from \eqref{rn} and  Lemma~\ref{Lemma:replacement-Lemmas-Memoirs}.
\end{proof}

We will need background on certain smooth polynomials defined in
terms of Hadamard products. If $W(z)=\sum_{k\in J}b_kz^k$ is a
polynomial and $f(z)=\sum_{k=0}^{\infty}a_kz^k\in \H(\D)$, then
the Hadamard product
    $$
    (W\ast f)(z)=\sum_{k\in J}b_ka_kz^k
    $$
is well defined.

If $\Phi:\mathbb{R}\to\C$ is a $C^\infty$-function such that its
support, $\supp(\Phi)$, is a compact subset of $(0,\infty)$, we set
    $$
    A_{\Phi,m}=\max_{s\in\mathbb{R}}|\Phi(s)|+\max_{s\in\mathbb{R}}|\Phi^{(m)}(s)|,
    $$
and consider the polynomials
    $$
    W_N^\Phi(z)=\sum_{k\in\mathbb
    N}\Phi\left(\frac{k}{N}\right)z^k,\quad N\in\N.
    $$
With this notation we can state the next result that follows by
\cite[p.~111--113]{Pabook}.

\begin{lettertheorem}\label{th:cesaro}
Let $\Phi:\mathbb{R}\to\C$ be a $C^\infty$-function such that
$\supp(\Phi)\subset(0, \infty)$ is compact. Then for each $p\in(0,\infty)$ and $m\in\N$ with $mp>1$, there exists a constant
$C=C(p)>0$ such that
    $$
    \|W_N^\Phi\ast f\|_{H^p}\le C A_{\Phi,m}\|f\|_{H^p}
    $$
for all $f\in H^p$ and $N\in\N$.
\end{lettertheorem}

For $g(z)=\sum_{k=0}^\infty b_k z^k\in\H(\D)$ and
$n_1,n_2\in\N\cup\{0\}$, we set
    $$
    S_{n_1,n_2}g(z)=\sum_{k=n_1}^{n_2-1}b_kz^k,\quad n_1<n_2,
    $$
and for each radial weight $\om$, we write
    $$
    \om_x=\int_0^1 r^{2x+1}\om(r)\,dr,\quad x>-1.
    $$

The next result is known and can be proved by summing by parts and using the M.~Riesz projection
theorem, see \cite[Lemma~E]{PelRathg}.

\begin{letterlemma}\label{le:A}
Let $1<p<\infty$ and
$\lambda=\left\{\lambda_k\right\}_{k=0}^\infty$ be a monotone
sequence of positive numbers. Let $(\lambda
g)(z)=\sum_{k=0}^{\infty} \lambda_kb_k z^k$, where
$g(z)=\sum_{k=0}^\infty b_k z^k$.
\begin{itemize}
\item[\rm(a)] If $\left\{\lambda_k\right\}_{n=0}^\infty$ is
nondecreasing, then there exists a constant $C>0$ such that
    $$
    C^{-1}\lambda_{n_1}\|S_{n_1,n_2} g\|_{H^p}\le\|S_{n_1,n_2}\lambda g\|_{H^p}
    \le C\lambda_{n_2}\|S_{n_1,n_2}g\|_{H^p}.
    $$
\item[\rm(b)] If $\left\{\lambda_n\right\}_{n=0}^\infty$ is
nonincreasing, then there exists a constant $C>0$ such that
    $$
    C^{-1}\lambda_{n_2}\|S_{n_1,n_2} g\|_{H^p}
    \le\|S_{n_1,n_2} \lambda g\|_{H^p}
    \le C\lambda_{n_1}\|S_{n_1,n_2} g\|_{H^p}.
    $$
\end{itemize}
 \end{letterlemma}

We will also need an extension of this result for $0<p\le 1$ in the case when $\lambda_k$ is either $\om_k$ or $\om_k^{-1}$.

\begin{lemma}\label{le:decre}
Let $0<p<\infty$, $\om$ a radial weight and
$n_1,n_2\in\N$ with $n_1<n_2$. Let $g(z)=\sum_{k=0}^{\infty}c_k
z^k$ be analytic in $\D$, and assume that both,
$h(z)=\sum_{k=0}^{\infty}c_k\om_kz^k$ and
$H(z)=\sum_{k=0}^{\infty}\frac{c_k}{\om_k}z^k$, are analytic in $\D$ as
well. Then the following assertions hold:
    \begin{enumerate}
    \item[(i)] There exists a constant $C=C(p)>0$ such that
    \begin{equation*}
    \|S_{n_1,n_2}h\|_{H^p} \le C\om_{\frac{n_1-1}{2}}\| g\|_{H^p}.
    \end{equation*}
    \item[(ii)] If $\om\in\DD$ and $n_1<n_2\le Kn_1$ for some $K>0$, then
    there exists a constant $C=C(p,\om,K)>0$ such that
    \begin{equation*}
    \|S_{n_1,n_2}H\|_{H^p} \le C\left(\om_{\frac{n_1-1}{2}}\right)^{-1}\| g\|_{H^p}.
    \end{equation*}
    \end{enumerate}
\end{lemma}

\begin{proof} (i). Define
    $$
    \Upsilon_{n_1}(s)=\int_0^1 r^{2n_1s+1}\om(r)\,dr,\quad s\ge 0.
    $$
Clearly, $\Upsilon_{n_1}$ is a $C^\infty$-function and
    \begin{equation}\label{eq:up1}
    \left| \Upsilon_{n_1}(s)\right|
    \le\int_0^1 r^{n_1}\om(r)\,dr,\quad s\ge1-\frac{1}{2n_1}.
    \end{equation}
Further, since
    $
    C(m)=\sup_{0<x<1}\left(\log
    \frac{1}{x}\right)^m x^{1/2}<\infty,
    $
we have
\begin{equation}
\begin{split}\label{eq:up2}
\left| \Upsilon^{(m)}_{n_1}(s)\right| &
\le \int_0^1\left[\left(\log\frac{1}{r^{2n_1}}\right)^m r^{n_1}\right]\,r^{2n_1s+1-n_1}\om(r)\,dr
\\  & \le C(m)\int_0^1 r^{n_1}\om(r)\,dr,\quad s\ge1-\frac{1}{2n_1}.
\end{split}
\end{equation}
Therefore, by using \eqref{eq:up1} and \eqref{eq:up2}, we can find
a function $\Phi_{n_1}\in C^\infty$ such that
$\supp(\Phi_{n_1})\in\left(1-\frac{1}{2n_1},
\frac{n_2}{n_1}\right)$,
    $$
    \Phi_{n_1}(s)= \Upsilon_{n_1}(s),\quad s\in\,\left[1,\frac{n_2-1}{n_1}\right],
    $$
and
    $$
    A_{\Phi_{n_1},m}=\max_{s\in\mathbb{R}}|\Phi_{n_1}(s)|+\max_{s\in\mathbb{R}}|\Phi_{n_1}^{(m)}(s)|\le
    C(m)\om_{\frac{n_1-1}{2}}.
    $$
Therefore we can write
    \begin{align*}
    S_{n_1,n_2}h(z)&=
    \sum_{k=n_1}^{n_2-1}c_k\om_kz^k
    \\ & =\sum_{k=n_1}^{n_2-1}c_k\Phi_{n_1}\left (\frac{k}{n_1}\right )z^k
    =\left(W_{n_1}^{\Phi_{n_1}}\ast g\right)(z),\quad z\in\T.
    \end{align*}
Hence, by fixing $m$ sufficiently large so that $mp>1$, and using Theorem~\ref{th:cesaro},
we obtain
    \begin{equation*}
    \begin{split}
    \| S_{n_1,n_2}h\|_{H^p}
    = \|W_{n_1}^{\Phi_{n_1}}\ast g\|_{H^p}
     \le C_2A_{\Phi_{n_1},m}\|g\|_{H^p}
    \le C(m)C_2 \om_{\frac{n_1-1}{2}}
    \|g\|_{H^p},
    \end{split}
    \end{equation*}
where $C_2=C_2(p)>0$ is a constant. Thus (i) is proved.

(ii). We set $\varphi_{n_1}(s)=\left(\Upsilon_{n_1}(s)\right)^{-1}$ and will prove that
    \begin{equation}\label{m}
    \begin{split}
    A_{\phi_{n_1},m}&=\max_{ 1-\frac{1}{2n_1}\le s\le \frac{n_2}{n_1}}|\phi_{n_1}(s)|
    +\max_{ 1-\frac{1}{2n_1}\le s\le \frac{n_2}{n_1}}|\phi_{n_1}^{(m)}(s)|\\
    &\le C(m,\om,K)\left(\om_{\frac{n_1-1}{2}}\right)^{-1},\quad m\in\N\cup\{0\}.
    \end{split}
    \end{equation}
Since
    $$
    \varphi_{n_1}(s)\le \frac{1}{\int_0^1 r^{2n_2+1}\om(r)\,dr},\quad 0\le s\le \frac{n_2}{n_1},
    $$
Lemma~\ref{Lemma:replacement-Lemmas-Memoirs} and the hypothesis $n_2\le K n_1$ yield
    \begin{equation}
    \begin{split}\label{eq:n1}
    \varphi_{n_1}(s)&\le\frac{\int_0^1 r^{n_1}\om(r)\,dr}{\int_0^1 r^{2n_2+1}\om(r)\,dr}\left( \om_{\frac{n_1-1}{2}}\right)^{-1}
    \asymp\frac{\widehat{\om}\left(1-\frac{1}{n_1}\right)}{\widehat{\om}\left(1-\frac{1}{2n_2+1}\right)}\left(\om_{\frac{n_1-1}{2}}\right)^{-1}\\
    &\lesssim\left(\frac{2n_2+1}{n_1}\right)^\b \left( \om_{\frac{n_1-1}{2}}\right)^{-1}
    \lesssim\left( \om_{\frac{n_1-1}{2}}\right)^{-1}
    \end{split}
    \end{equation}
for all $0\le s\le \frac{n_2}{n_1}$, where $\beta=\beta(\om)\in (0,\infty)$. This gives \eqref{m} for
$m=0$.

If $m=1$, we may use \eqref{eq:up2} and \eqref{eq:n1} to obtain
    $$
    |\varphi'_{n_1}(s)|=\frac{|\Upsilon'_{n_1}(s)|}{|\Upsilon_{n_1}(s)|^2}=|\Upsilon'_{n_1}(s)||\varphi_{n_1}(s)|^2 \lesssim\left( \om_{\frac{n_1-1}{2}}\right)^{-1}
    $$
for all $1-\frac{1}{2n_1}\le s\le \frac{n_2}{n_1}$. The general
case is now proved by induction. Assume that \eqref{m} holds
for $j=1,\dots,m-1$, where $m>1$. Since
$1=\varphi_{n_1}(s)\Upsilon_{n_1}(s)$, we have
    $$
    0=(\varphi_{n_1}\Upsilon_{n_1})^{(m)}(s)=\sum_{j=0}^m
    \binom{m}{j} \Upsilon_{n_1}^{(m-j)}(s)\varphi^{(j)}_{n_1}(s),
    $$
which implies
    $$
    |\varphi_{n_1}^{(m)}(s)|\le\frac{\sum_{j=0}^{m-1}\binom{m} {j}\left|\Upsilon_{n_1}^{(m-j)}(s)\varphi^{(j)}_{n_1}(s)\right|}{|\Upsilon_{n_1}(s)|}.
    $$
This together with the induction hypothesis and \eqref{eq:up2} gives \eqref{m}. The proof can be completed arguing as in (i). We omit the details.
\end{proof}

We now turn to $L^p$-estimates. For that purpose we will use the fact that if $\{e_n\}$ is an orthonormal basis of a Hilbert
space $H$, that is continuously embedded into $\H(\D)$, then its reproducing kernel is given by
    \begin{equation}\label{RKformula}
    K_z(\zeta)=\sum_n e_ n(\zeta)\,\overline{e_ n(z)}
    \end{equation}
for all $z$ and $\zeta$ in $\D$, see~\cite[Theorem~4.19]{Zhu}. We shall write  $\omega_\b(z)=(1-|z|)^\b\omega(z)$ for all
$\b\in\mathbb{R}$ and $z\in\D$.

\begin{lemma}\label{HL}
Let $\om\in\DD$ and $n\in\N\cup\{0\}$.
\begin{enumerate}
\item[\rm(i)] If $0<p\le2$, then
    $$
  M_p^p\left(r,\left(B^\om_a\right)^{(N)}\right)\gtrsim
     \int_0^{|a|r}\frac{dt}{\widehat{\om}(t)^{p}(1-t)^{p(N+1)}},\quad r,|a|\to1^-.
    $$
  \item[\rm(ii)] If $2\le p<\infty$, then
    $$
    M_p^p\left(r,\left(B^\om_a\right)^{(N)}\right)\lesssim
     \int_0^{|a|r}\frac{dt}{\widehat{\om}(t)^{p}(1-t)^{p(N+1)}},\quad r,|a|\to1^-.
    $$
\end{enumerate}
\end{lemma}

\begin{proof}
By using the standard orthonormal basis
$\{z^{j}/\sqrt{2\om_{j}}\}$, $j\in \N\cup\{0\}$, of
$A^2_\om$ and \eqref{RKformula} we obtain
    \begin{equation}\label{kernelformula}
    B^\om_a(z)=\sum_{n=0}^\infty\frac{(z\overline{a})^n}{2\om_n},
    \end{equation}
which implies
    \begin{equation*}
    \left(B^\om_a\right)^{(N)}(z)=\sum_{j= N}^\infty
    \frac{j(j-1)\cdots(j-N+1)z^{j-N}\overline{a}^{j}}{2\om_{j}},\quad n\in\N.
    \end{equation*}
Therefore the classical Hardy-Littlewood inequalities \cite[Theorem~6.2]{Duren1970} applied to the dilated function show that it suffices to prove
    \begin{equation}\label{series}
    \sum_{n=N}^\infty\frac{r^{pn}}{(n+1)^{-(N+1)p+2}\om_{n}^{p}}\asymp \int_0^{r}\frac{dt}{\widehat{\om}_{N+1}(t)^{p}},
    \quad r \to 1^-.
    \end{equation}
Assume, without loss of generality, that $r>1-\frac1{N+1}$. Choose now $N^\star\in\N$ such that $1-\frac{1}{N^\star}\le r<1-\frac{1}{N^\star+1}$. Then Lemma~\ref{Lemma:replacement-Lemmas-Memoirs} yields
    \begin{equation*}
    \begin{split}
    \sum_{n=N}^{N^\star}\frac{r^{pn}}{(n+1)^{-(N+1)p+2}\om_{n}^{p}}
    &\asymp \sum_{n=N}^{N^\star}\frac{1}{(n+1)^{-(N+1)p+2}\om_{n}^{p}}\\
    &\asymp\sum_{n=N}^{N^\star}\frac{1}{(n+1)^{-N(p+1)+2}\widehat{\om}(1-\frac{1}{2n+1})^{p}}\\
    &\gtrsim\int_{N+1}^{N^\star+1}\frac{ds}{s^{-N(p+1)+2}\widehat{\om}(1-\frac{1}{s})^{p}}\\
    & \ge \int_{N+1}^{\frac{1}{1-r}}\frac{ds}{s^{-N(p+1)+2}\widehat{\om}(1-\frac{1}{s})^{p}}\\
    &=\int_{1-\frac{1}{N+1}}^{r}\frac{dt}{\widehat{\om}_{N+1}(t)^{p}}\asymp \int_{0}^{r}\frac{dt}{\widehat{\om}_{N+1}(t)^{p}},\quad r \to 1^-.
    \end{split}
    \end{equation*}
Lemma~\ref{Lemma:replacement-Lemmas-Memoirs} allows us to establish the same upper bound in a similar manner, so
  \begin{equation*}
    \begin{split}
    \sum_{n=N}^{N^\star}\frac{r^{pn}}{(n+1)^{-(N+1)p+2}\om_{n}^{p}}
   \asymp \int_{0}^{r}\frac{dt}{\widehat{\om}_{N+1}(t)^{p}},\quad r\to 1^-.
    \end{split}
    \end{equation*}
Lemma~\ref{Lemma:replacement-Lemmas-Memoirs} also implies
    \begin{equation}\label{bi}
    \frac{1}{\widehat{\om}(r)^{p}(1-r)^{p(N+1)-1}}
    \asymp\int_{\frac{4r-1}{3}}^r  \frac{dt}{\widehat{\om}_{N+1}(t)^{p}},\quad r\to 1^-,
    \end{equation}
and the existence of $M=M(p,\om)>1$ such that $\frac{\widehat{\om}(r)^p}{(1-r)^M}$ is essentially increasing. Hence
    \begin{equation*}
    \begin{split}
    \sum_{n=N^\star}^\infty\frac{r^{pn}}{(n+1)^{-(N+1)p+2}\om_{n}^{p}}
    &\asymp \sum_{n=N^\star}^\infty \frac{r^{n}}{(n+1)^{-N(p+1)+2}\widehat{\om}(1-\frac{1}{n})^{p}}
    \\ & \lesssim \frac{1}{(N^\star+1)^M \widehat{\om}(1-\frac{1}{N^\star})^{p}}
    \sum_{n=N^\star}^\infty (n+1)^{p(N+1)-2+M}r^{n}
    \\ & \asymp \frac{ (N^\star+1)^{p(N+1)-1}}{ \widehat{\om}(1-\frac{1}{N^\star})^{p}}
    \asymp\frac{1}{(1-r)^{p(N+1)-1}\widehat{\om}(r)^{p}}
   \\ &  \asymp\int_{\frac{4r-1}{3}}^r  \frac{dt}{\widehat{\om}_{N+1}(t)^{p}}
    \le \int_{0}^{r}\frac{dt}{\widehat{\om}_{N+1}(t)^{p}},\quad r\to 1^-,
    \end{split}
    \end{equation*}
and the proof is complete.
\end{proof}

\subsection{Proof of Theorem~\ref{th:kernelstimate}.}

We begin with proving (ii) for $v\in\R$ continuous. In this case we have two advantages compared to the general case $v\in\DD$. First, the main result in
~\cite{PavP} implies the Littlewood-Paley formula
    \begin{equation}\label{LPformula}
    \|f\|_{A^p_v}^p\asymp\int_{\D}|f^{(n)}(z)|^p(1-|z|)^{np}v(z)\,dA(z)+\sum_{j=0}^{n-1}|f^{(j)}(0)|^p,\quad f\in \H(\D),
    \end{equation}
for all $0<p<\infty$, $v\in\R$ and $n\in\N$. This allows us to assume that the order $N$ of the derivative is sufficiently large, and in that way we avoid some difficulties in the proof. Note that \eqref{LPformula} fails in general for $v\in\I\subset\DD$ by~\cite[Proposition~4.3]{PelRat}. Second, when $v\in\R$, we have precise control over the size of the blocks $\Delta^v_n f$
appearing in the decomposition of the $A^p_v$-norm of $f$.

\medskip

{\bf{ Part~(ii). Case~$\mathbf{v\in\R}$} continuous.}
By the Littlewood-Paley formula \eqref{LPformula} we may assume that $N>\frac{1}{p}-1$.
Without loss of generality, we may also assume $\int_0^1 v(r)\,dr=1$.
Therefore Lemma~\ref{le:de1} (the case $p\le1$) or
\cite[Theorem~4]{PelRathg} (the case $p>1$),
\eqref{kernelformula} and Lemma~\ref{le:A} give
    \begin{equation}
    \begin{split}\label{eq:ke1}
    &\int_{\D}\left|\left(B^\om_a\right)^{(N)}(z)\right|^p v(z)\,dA(z)
    =\int_{\D}\left|\left(B^\om_{|a|}\right)^{(N)}(z)\right|^p v(z)\,dA(z)\\
    &\lesssim\sum_{n=0}^\infty 2^{-n}\left\|\Delta^{v}_n
    \left(B^\om_{|a|}\right)^{(N)}\right\|_{H^p}^p\\
    &=\sum_{n=0}^\infty 2^{-n}\left\|\sum_{j\in I_{v}(n),j\ge N}
    \frac{j(j-1)\cdots(j-N+1)z^{j-N}|a|^{j}}{2\om_j}\right\|_{H^p}^p\\
    &\lesssim\sum_{n=0}^\infty 2^{-n}|a|^{M_n}\left\|\sum_{j\in I_{v}(n),j\ge N}
    \frac{j(j-1)\cdots(j-N+1)z^{j-N}}{2\om_j}\right\|_{H^p}^p.
    \end{split}
    \end{equation}
Now, since $v$ is regular, \cite[Lemma~6]{PelRathg} implies $\sup_{n\ge 0}\frac{M_{n+1}}{M_n}<\infty$, where $M_n=E\left(\frac{1}{1-r_n}\right)$ are associated to $v$ via $\widehat{v}(r_n)=2^{-n}$. So, by using this, Lemma~\ref{le:decre}(ii), \cite[Lemma~10]{PelRathg} and the assumption $N>\frac{1}{p}-1$, we get
    \begin{equation}\label{11111}
    \begin{split}
    &\left\|\sum_{j\in I_{v}(n),j\ge N}
    \frac{j(j-1)\cdots(j-N+1)z^{j-N}}{2\om_j}\right\|_{H^p}^p
    \\ & \lesssim
    \frac{1}{\left(\om_{\frac{M_n-1}{2}}\right)^p}
    M^p_p\left( 1-\frac{1}{M_{n+1}},\frac{d^{(N)}}{dz^{N}}
    \left(\frac{1}{1-z}\right)\right)
    \\ & \lesssim  \frac{1}{\left(\om_{\frac{M_n-1}{2}}\right)^p}
    M^p_p\left( 1-\frac{1}{M_{n+1}},\frac{1}{(1-z)^{N+1}}\right)
    \\ & \lesssim \frac{M^{(N+1)p-1}_{n+1}}{\left(\om_{\frac{M_n-1}{2}}\right)^p}
    \lesssim \frac{M^{(N+1)p-1}_{n}}{\left(\om_{\frac{M_n-1}{2}}\right)^p}.
    \end{split}
    \end{equation}
Next, \eqref{LP1} for $f(z)=z^{n}$ and Lemma~\ref{Lemma:replacement-Lemmas-Memoirs} applied to $\om^\star\in\DD$ yield
    \begin{equation}\label{moments}
    \om_{n}=4n^2\om^\star_{n-1}\asymp n^2\om^\star_{n},
    \end{equation}
which together with \eqref{eq:ke1} and \eqref{11111} implies
    \begin{equation}\label{eq:ke2}
    \int_{\D}\left|\left(B^\om_a\right)^{(N)}(z)\right|^p v(z)\,dA(z)
    \lesssim\sum_{n=0}^\infty \frac{2^{-n}M^{(N-1)p-1}_{n}}
    {\left(\om^\star_{\frac{M_n-1}{2}}\right)^p}|a|^{M_n}.
    \end{equation}
The last step in this part of the proof consists of bounding the series in
\eqref{eq:ke2}.

Since $v$ is a regular weight, the definition $\widehat{v}(r_n)=2^{-n}$ and Lemma~\ref{Lemma:replacement-Lemmas-Memoirs} imply $v^\star (r_n)\asymp 2^{-n}M_n^{-1}$.
Moreover, since $\om^\star\in\DD$, Lemma~\ref{Lemma:replacement-Lemmas-Memoirs} and $\om^\star(z)\asymp\widehat{\om}(z)(1-|z|)$ yield
   \begin{equation}\label{eqmo}
    \om^\star_{\frac{M_n-1}{2}}\asymp \widehat{\om^\star}\left(1-\frac1{M_n}\right)
    \asymp \om^\star\left(1-\frac{1}{M_n}\right)M_n^{-1}.
    \end{equation}
Therefore, by using \cite[Lemma~6]{PelRathg} and Lemma~\ref{Lemma:replacement-Lemmas-Memoirs}, we deduce
    \begin{equation}
    \begin{split}\label{eq:ke3}
    &\sum_{n=0}^\infty \frac{2^{-n}M^{(N-1)p-1}_{n} }{\left(\om^\star_{\frac{M_n-1}{2}}\right)^p}|a|^{M_n}
      \asymp 1+\sum_{n=1}^\infty \frac{v^\star (r_n)M^{Np}_{n} }{\left( \om^\star\left(1-\frac{1}{M_n}\right)\right)^p}|a|^{M_n}\\
      &\asymp 1+\sum_{n=1}^\infty \frac{v^\star \left(1-\frac{1}{M_n}\right)M^{Np-1}_{n} }{\left( \om^\star\left(1-\frac{1}{M_n}\right)\right)^p}\left(M_n-M_{n-1}\right)|a|^{M_n}
  \\  &\le 1+\sum_{n=1}^\infty \frac{v^\star \left(1-\frac{1}{M_n}\right)M^{Np-1}_{n}}{\left( \om^\star\left(1-\frac{1}{M_n}\right)\right)^p}
    \sum_{j\in I_{v}(n-1)}|a|^{j}\\
      &\asymp 1+\sum_{n=1}^\infty
    \sum_{j\in I_{v}(n-1)}  \frac{v^\star \left(1-\frac{1}{j+1}\right)(j+1)^{Np-1}}{\left( \om^\star\left(1-\frac{1}{j+1}\right)\right)^p} |a|^{j}
    \\ & \asymp 1+\sum_{j=1}^\infty \frac{v^\star \left(1-\frac{1}{j+1}\right)(j+1)^{Np-1}}{\om^\star\left(1-\frac{1}{j+1}\right)^p} |a|^{j}.
    \end{split}
    \end{equation}
Let now $|a|\ge\frac34$. We observe that Lemma~\ref{Lemma:replacement-Lemmas-Memoirs} imply
    \begin{equation}\begin{split}\label{domi}
    \frac {v^\star \left(a\right)}{\om^\star(a)^{p}(1-|a|)^{Np}}&\asymp
    \int_{2|a|-1}^{|a|}\frac{\widehat{v}\left(s\right)}{\widehat{\om}(s)^p(1-s)^{(N+1)p}}\,ds
    \le \int_{0}^{|a|}\frac{\widehat{v}\left(s\right)}{{\widehat{\om}_{N+1}(s)^{p}}}\,ds.
    \end{split}\end{equation}
Next, take $N^\star\in\N$ such that $1-\frac{1}{N^\star}\le
|a|<1-\frac{1}{N^\star+1}$. Then, by \eqref{domi},
    \begin{equation}
    \begin{split}\label{eq:ke4}
    \sum_{j=1}^{N^\star} \frac{v^\star \left(1-\frac{1}{j+1}\right)(j+1)^{Np-1}}{\om^\star\left(1-\frac{1}{j+1}\right)^p} |a|^{j}
    & \lesssim \int_2^{\frac1{1-|a|}+2}\frac{v^\star\left(1-\frac1x\right)}{\om^\star\left(1-\frac1x\right)^p}x^{Np-1}\,dx\\
    &=\int_{1/2}^{\frac{1+|a|}{2}}\frac{v^\star\left(s\right)}{\om^\star\left(s\right)^p(1-s)^{Np+1}}\,ds\\
   &\lesssim\int_0^{|a|}\frac{\widehat{v}\left(s\right)}{\widehat{\om}_{N+1}(s)^{p}}\,ds
    \end{split}
    \end{equation}
for all $|a|\ge\frac34$. On the other hand, the function $h(r)=\widehat\om(r)(1-r)^{-\b}$ is essentially increasing on $[0,1)$ for $\b=\b(\om)$ sufficiently large by Lemma~\ref{Lemma:replacement-Lemmas-Memoirs}, and therefore
    $$
    \left\{(j+1)^{1+\b}\om^\star \left(1-\frac{1}{j+1}\right)\right\}_{j=1}^\infty
    $$
is an essentially increasing sequence. This and \eqref{domi} together with the fact that
    $$
    \left\{(j+1)v^\star \left(1-\frac{1}{j+1}\right)\right\}_{j=1}^\infty
    $$
is essentially decreasing, give
    \begin{equation*}
    \begin{split}
    &\sum_{j=N^\star+1}^\infty \frac{v^\star \left(1-\frac{1}{j+1}\right)(j+1)^{Np-1}}{\om^\star\left(1-\frac{1}{j+1}\right)^p} |a|^{j}
    \\ & \lesssim v^\star \left(1-\frac{1}{N^\star+2}\right)(N^\star+2)
    \sum_{j=N^\star+1}^\infty \frac{(j+1)^{Np-2}}{\om^\star\left(1-\frac{1}{j+1}\right)^p} |a|^{j}
    \\ & \lesssim \frac{v^\star \left(1-\frac{1}{N^\star+2}\right)(N^\star+2)^{1-\left(1+\b\right)p}}{\om^\star \left(1-\frac{1}{N^\star+2}\right)^p}
    \sum_{j=N^\star+1}^\infty (j+1)^{(N+1+\b)p-2} |a|^{j}
    \\ & \asymp \frac{v^\star \left(a\right)}{(1-|a|)^{1-\left(1+\b\right)p}\om^\star \left(a\right)^p}
    \sum_{j=N^\star+1}^\infty (j+1)^{(N+1+\b)p-2} |a|^{j}
    \\ &  \asymp \frac {v^\star \left(a\right)}{(1-|a|)^{Np}\left(\om^\star(a)\right)^{p}}
    \lesssim \int_0^{|a|}\frac{\widehat{v}\left(s\right)}{\widehat{\om}_{N+1}(s)^{p}}\,ds,
    \end{split}
    \end{equation*}
where in the last asymptotic equality we used our choice $N>\frac{1}{p}-1$.
This combined with \eqref{eq:ke2}, \eqref{eq:ke3} and
\eqref{eq:ke4} finishes the proof of the upper bound in \eqref{k1}, when $v\in\R$ is continuous.
\smallskip

In order to establish the same lower estimate, we will consider the cases $p>1$ and $0< p\le 1$ separately.
Let first $p>1$. By \cite[Theorem~4]{PelRathg}, \eqref{kernelformula}, Lemma~\ref{le:A},  \cite[Lemma~10]{PelRathg} and \eqref{moments} we deduce
    \begin{equation*}
    \begin{split}
       &\int_{\D}\left|\left(B^\om_a\right)^{(N)}(z)\right|^p v(z)\,dA(z)
    =\int_{\D}\left|\left(B^\om_{|a|}\right)^{(N)}(z)\right|^p v(z)\,dA(z)\\
    &\asymp\sum_{n=0}^\infty 2^{-n}\left\|\Delta^{v}_n
    \left(B^\om_{|a|}\right)^{(N)}\right\|_{H^p}^p\\
    &=\sum_{n=0}^\infty 2^{-n}\left\|\sum_{j\in I_{v}(n),j\ge N}
    \frac{j(j-1)\cdots(j-N+1)z^{j-N}|a|^{j}}{2\om_j}\right\|_{H^p}^p\\
    &\gtrsim\sum_{n=0}^\infty 2^{-n}|a|^{M_{n+1}}\left\|\sum_{j\in I_{v}(n),j\ge N}
    \frac{j(j-1)\cdots(j-N+1)z^{j-N}|a|^{j}}{2\om_j}\right\|_{H^p}^p
   \\ &\gtrsim\sum_{n=0}^\infty 2^{-n}|a|^{M_{n+1}}\frac{1}{\left(\om_{M_n}\right)^p}
     M^p_p\left( 1-\frac{1}{M_{n+1}},\frac{1}{(1-z)^{N+1}}\right)\\
    &\asymp\sum_{M_n\ge N}^\infty 2^{-n}|a|^{M_{n+1}}\frac{M^{(N-1)p-1}_{n+1}}{\left(\om^\star_{M_n}\right)^p}
    \asymp \sum_{M_{n-1}\ge N} \frac{2^{-n}M^{(N-1)p-1}_{n}}
   {\left(\om^\star_{M_n}\right)^p}|a|^{M_{n}},
    \end{split}
    \end{equation*}
where in the last step we have used the fact $M_n\asymp M_{n+1}$ and Lemma~\ref{Lemma:replacement-Lemmas-Memoirs} for $\om^\star\in\DD$.
Next, by using \eqref{eqmo}, \cite[Lemma~6]{PelRathg}, Lemma~\ref{Lemma:replacement-Lemmas-Memoirs}
and arguing in a manner similar to \eqref{eq:ke3}, we deduce
    \begin{equation*}
    \begin{split}\label{eq:ke3r}
    &\sum_{M_{n-1}\ge N} \frac{2^{-n}M^{(N-1)p-1}_{n} }{\left(\om^\star_{M_n}\right)^p}|a|^{M_n}
   \asymp \sum_{j\ge N} \frac{v^\star \left(1-\frac{1}{j+1}\right)(j+1)^{Np-1}}{\om^\star\left(1-\frac{1}{j+1}\right)^p} |a|^{j}.
    \end{split}\end{equation*}
Without loss of generality, we may assume $|a|>\max\{1-\frac1{N-1},\frac34\}$. Take $N^\star\in\N$ such that $1-\frac{1}{N^\star}\le
|a|<1-\frac{1}{N^\star+1}$. Then, arguing in a way similar to \eqref{eq:ke4} and bearing in mind \eqref{domi}, we get
 \begin{equation*}
    \begin{split}
    \sum_{j=N}^{N^\star+1} \frac{v^\star \left(1-\frac{1}{j+1}\right)(j+1)^{Np-1}}{\om^\star\left(1-\frac{1}{j+1}\right)^p} |a|^{j}
   \gtrsim\int_0^{|a|}\frac{\widehat{v}\left(s\right)}{\widehat{\om}_{N+1}(s)^{p}}\,ds,
    \end{split}
    \end{equation*}
and the desired lower bound follows.

Finally, let $0<p\le 1$. Then, by Lemma~\ref{HL}, Fubini's theorem, Lemma~\ref{Lemma:replacement-Lemmas-Memoirs} and
\eqref{domi}, we get
    \begin{equation}
    \begin{split}\label{eqn1}\int_{\D}\left|\left(B^\om_a\right)^{(N)}(z)\right|^p v(z)\,dA(z)
    & \gtrsim \int_0^1 \left(\int_{0}^{s|a|} \frac{dt}{\widehat{\om}_{N+1}(t)^{p}}\right) v(s)\,ds
    \\ & =\int_0^{|a|}\frac{\widehat{v}\left(\frac{t}{|a|}\right)}{\widehat{\om}_{N+1}(t)^{p}}\,dt\\
    &\gtrsim\int_0^{2|a|-1}\frac{\widehat{v}(t)}{\widehat{\om}_{N+1}(t)^{p}}\left(\frac{1-\frac t{|a|}}{1-t}\right)^\b\,dt\\
    &\gtrsim\int_0^{2|a|-1}\frac{\widehat{v}(t)}{\widehat{\om}_{N+1}(t)^{p}}\,dt\\
    &\asymp\int_0^{|a|}\frac{\widehat{v}(t)}{\widehat{\om}_{N+1}(t)^{p}}\,dt,\quad |a|\to 1^-.
    \end{split}
    \end{equation}
 Theorem~1(ii) for $v\in\R$ continuous is now proved.

\smallskip

Before proving (ii) for $v\in\DD$, we will prove (i).
To do this we will use the well known 
inclusions
    \begin{equation}\label{2}
    \mathcal{D}^p_{p-1}\subsetneq H^p,\quad0<p<2,
    \end{equation}
and
    \begin{equation}\label{1}
    H\sp p\subsetneq\mathcal{D}^p_{p-1},\quad2<p<\infty,
    \end{equation}
where $\mathcal{D}^p_{p-1}$ denotes the space of  $f\in\H(\D)$ such that
    $
    \int_\D|f'(z)|^p(1-|z|)^{p-1}\,dA(z)<\infty.
    $

\medskip

{\bf{Part~(i). Case~$\mathbf{0<p\le 2}$}}. Let $r\in[\frac12,1)$. Then, by \eqref{2},
and Theorem~\ref{th:kernelstimate}(ii) for the regular weight $v(z)=(1-|z|)^{p-1}$,
    \begin{equation*}
    \begin{split}
    M_p^p\left(\left(B^\om_a\right)^{(N)},r\right)&=\left\|\left(B^\om_a\right)^{(N)}_r\right\|_{H^p}^p
    \lesssim1+\int_\D\left|\frac{\partial^{N+1}B^\om_a(z,ra)}{\partial^{N+1}z}\right|^p(1-|z|)^{p-1}\,dA(z)\\
    &\asymp 1+\int_0^{r|a|}\frac{ds}{\widehat{\om}_{N+1}(s)^{p}},\quad |a|\to 1^-.
    \end{split}
    \end{equation*}
The reverse implication follows by Lemma~\ref{HL}(i).
\medskip

{\bf{Part~(i). Case~$\mathbf{2<p<\infty}$}}. It can be proved similarly, by using \eqref{1}, Theorem~\ref{th:kernelstimate}(ii) for $v(z)=(1-|z|)^{p-1}$ and Lemma~\ref{HL}(ii).

\medskip

The proof of Theorem~\ref{th:kernelstimate}(i) is now complete.

\medskip

{\bf{Part~(ii). Case~$\mathbf{v\in\DD}$}}. If $v$ is a radial weight, then Theorem~\ref{th:kernelstimate}(i) and
Fubini's theorem yield
    \begin{equation}
    \begin{split}\label{anyradial}
    \int_\D\left|\left(B^\om_a\right)^{(N)}(z)\right|^pv(z)\,dA(z)
    & \asymp \int_0^1 \left(\int_{0}^{s|a|} \frac{dt}{\widehat{\om}_{N+1}(t)^{p}}\right) v(s)\,ds
    \\ & =\int_0^{|a|}\frac{\widehat{v}\left(\frac{t}{|a|}\right)}{\widehat{\om}_{N+1}(t)^{p}}\,dt
    \\ & \le \int_0^{|a|}\frac{\widehat{v}\left(t\right)}{\widehat{\om}_{N+1}(t)^{p}}\,dt,\quad |a|\ge\frac12.
     \end{split}
    \end{equation}
The reverse inequality for $v\in\DD$ can be proved by combining \eqref{anyradial} with the argument used in \eqref{eqn1}.
This finishes the proof of Theorem~\ref{th:kernelstimate}(ii).\hfill$\Box$

\medskip

\noindent\emph{Proof of Corollary~\ref{co:ernelstimaten}.} The equivalence between the asymptotic equality (i) and \eqref{Eq:hypothesis-kernelmeans} follows by Theorem~\ref{th:kernelstimate}(i) and \eqref{bi}. Moreover, \eqref{kn4} is equivalent to \eqref{Eq:hypothesis-kernel} by Theorem~\ref{th:kernelstimate}(ii) and \eqref{domi}.\hfill$\Box$

\section{Projections}\label{sec:projections}

\subsection{Two weight inequality}

Theorem~\ref{theorem:projections2} is contained in the following result.

\begin{theorem}\label{theorem:projections3}
Let $1<p<\infty$ and $\om,v\in\R$. Then the following conditions are equivalent:
\begin{itemize}
\item[\rm(a)] $P^+_\om:L^p_v\to L^p_v$ is bounded;
\item[\rm(b)] $P_\om:L^p_v\to L^p_v$ is bounded;
\item[\rm(c)]  $\displaystyle
\sup_{0<r<1} \frac{\widehat{v}(r)^{\frac{1}{p}}
    \left(\int_r^1\left(\frac{\om(s)}{v(s)}\right)^{p'}\,v(s)ds\right)^{\frac{1}{p'}}
    }{\widehat\om(r)}<\infty$;
\item[\rm(d)] $\displaystyle \sup_{0<r<1}\frac{\om(r)^p(1-r)^{p-1}}{v(r)}\int_0^r\frac{v(s)}{\om(s)^p(1-s)^p}\,ds<\infty$;
\item[\rm(e)] $\displaystyle \sup_{0<r<1}\left(\int_0^r\frac{v(s)}{\om(s)^p(1-s)^p}\,ds\right)^{\frac{1}{p}}
    \left(\int_r^1\left(\frac{\om(s)}{v(s)}\right)^{p'}\,v(s)ds\right)^{\frac{1}{p'}}<\infty;$
\item[\rm(f)] $\displaystyle
\sup_{0<r<1} \frac{\widehat{v}(r)^{\frac{1}{p}}
    \int_r^1\frac{\om(s)}{\left((1-s)v(s)\right)^{1/p}}\,ds
    }{\widehat\om(r)}<\infty$;
\item[\rm(g)]$\displaystyle \sup_{0<r<1}\frac{\om(r)(1-r)^{\frac{1}{p'}}}{v(r)^{1/p}}\int_0^r\frac{v(s)^\frac{1}{p}}{\om(s)(1-s)^{1+\frac{1}{p'}}}\,ds<\infty$.
\end{itemize}
\end{theorem}

\begin{proof}
The implication (a)$\Rightarrow$(b) is obvious, so assume (b). A direct calculation shows that
the adjoint of $P_\om$, with respect to $\langle \cdot,\cdot\rangle_{L^2_v}$, is given by
    \begin{equation}\label{eq:adjoint}
    P^\star_\om(g)(\z)=\frac{\om(\z)}{v(\z)}\int_{\D}g(z)B^\om(\z,z)v(z)\,dA(z),\quad g\in L^{p'}_v.
    \end{equation}
By the hypothesis, $P_\om^\star:L^{p'}_v\to L^{p'}_v$ is bounded, and hence, by choosing $g_n(z)=z^n$, $n\in\N$, and using \eqref{RKformula} with the standard basis of $A^2_\om$, we deduce
    \begin{equation*}
    \begin{split}
    \left(\frac{v_n}{\om_n}\right)^{p'}\int_\D|\z|^{np'}\left(\frac{\om(\z)}{v(\z)}\right)^{p'}v(\z)\,dA(\z)
    \asymp\|P_\om^\star(g_n)\|_{L^{p'}_v}^{p'}\lesssim\|g_n\|_{L^{p'}_v}^{p'}=2v_{\frac{np'}{2}}.
    \end{split}
    \end{equation*}
This together with Lemma~\ref{Lemma:replacement-Lemmas-Memoirs} gives
    \begin{equation*}
    \begin{split}
    \infty&>\sup_{n}\left(\frac{v_n^{p'}}{v_{\frac{np'}{2}}\om^{p'}_n}\int_0^1 s^{np'+1}\left(\frac{\om(s)}{v(s)}\right)^{p'}v(s)\,ds\right)\\
    &\asymp\sup_{n}\left(\frac{\widehat{v}\left(1-\frac{1}{n}\right)^{p'-1}}{\widehat{\om}\left(1-\frac{1}{n}\right)^{p'}}
    \int_0^1 s^{np'+1}\left(\frac{\om(s)}{v(s)}\right)^{p'}v(s)\,ds\right)\\
    &\gtrsim\sup_{n}\left(\frac{\widehat{v}\left(1-\frac{1}{n}\right)^{p'-1}}{\widehat{\om}\left(1-\frac{1}{n}\right)^{p'}}
    \int_{1-\frac1n}^1\left(\frac{\om(s)}{v(s)}\right)^{p'}v(s)\,ds\right).
    \end{split}
    \end{equation*}
By choosing $r\in[0,1)$ such that $r\in [1-\frac{1}{n},1-\frac{1}{n+1})$, we deduce
    \begin{equation*}
    \begin{split}
    \sup_{r\in[0,1)}\frac{\widehat{v}(r)^{p'-1}}{\widehat{\om}(r)^{p'}}\int_{r}^1\left(\frac{\om(s)}{v(s)}\right)^{p'}v(s)\,ds<\infty,
    \end{split}
    \end{equation*}
which is equivalent to (c).

(c)$\Rightarrow$(a). Let $h(r)=v^{1/p}(r)\left(\int_r^1\left(\frac{\om(s)}{v(s)}\right)^{p'}\,v(s)ds\right)^{\frac{1}{pp'}}$.
By the hypothesis (c),
    \begin{equation}
    \begin{split}\label{eq:mu1}
    \int_{t}^1 \left(\frac{\om(s)}{h(s)}\right)^{p'}\,ds
    =p'\left(\int_r^1\left(\frac{\om(s)}{v(s)}\right)^{p'}\,v(s)ds\right)^{\frac{1}{p'}}
    \lesssim \frac{\widehat{\om}(r)}{\widehat{v}(r)^{1/p}}.
    \end{split}
    \end{equation}
H\"older's inequality yields
    \begin{equation}
    \begin{split}\label{eq:mu2}
    \|P^+_\om(f)\|^p_{L^p_v}
    &\le\int_{\D}\left(\int_{\D}|f(\z)|^ph(\z)^p|B^\om(z,\z)|\,dA(\z)\right)\\
    &\quad\cdot\left(\int_{\D}|B^\om(z,\z)|\left(\frac{\om(\z)}{h(\z)}\right)^{p'}\,dA(\z)\right)^{p/p'}v(z)\,dA(z),
    \end{split}
    \end{equation}
where, by Theorem~\ref{th:kernelstimate}(i) and \eqref{eq:mu1},
    \begin{equation}
    \begin{split}\label{eq:mu3}
    \int_{\D}|B^\om(z,\z)|\left(\frac{\om(\z)}{h(\z)}\right)^{p'}\,dA(\z)
    &\lesssim \int_0^{1}\left(\frac{\om(s)}{h(s)}\right)^{p'}\left(\int_0^{s|z|}\frac{dt}{\widehat{\om}(t)(1-t)}\right)\,ds
    \\ & =  \int_0^{|z|}\left(\int_{t/|z|}^1\left(\frac{\om(s)}{h(s)}\right)^{p'}\,ds\right)\frac{dt}{\widehat{\om}(t)(1-t)}
    \\ & \le\int_0^{|z|}\left(\int_{t}^1\left(\frac{\om(s)}{h(s)}\right)^{p'}\,ds\right)\frac{dt}{\widehat{\om}(t)(1-t)}
    \\ & \lesssim
    \int_0^{|z|}\frac{dt}{\widehat{v}(t)^{1/p}(1-t)}.
    \end{split}
    \end{equation}
Since $v\in\R$, Lemma~\ref{Lemma:Muckenhoupt} below, with $\alpha=1+\frac{1}{p}$, gives
    \begin{equation*}\begin{split}
    \int_0^{|z|}\frac{dt}{(1-t)\widehat{v}(t)^{1/p}}
   \asymp \frac{1}{\widehat{v}(z)^{1/p}},
  \end{split}\end{equation*}
which combined with \eqref{eq:mu3} yields
    \begin{equation*}
    \begin{split}
    &\int_{\D}|B^\om(z,\z)|\left(\frac{\om(\z)}{h(\z)}\right)^{p'}\,dA(\z)\lesssim \frac{1}{\widehat{v}(z)^{1/p}}.
    \end{split}
    \end{equation*}
This together with \eqref{eq:mu2} and Fubini's theorem give
    \begin{equation}
    \begin{split}\label{eq:mu4}
    \|P^+_\om(f)\|^p_{L^p_v}
    &\lesssim
    \int_{\D}\left(\int_{\D}|f(\z)|^ph(\z)^p|B^\om(z,\z)|\,dA(\z)\right)\frac{v(z)}{\widehat{v}(z)^{1/p'}}\,dA(z)\\
    &=\int_{\D}|f(\z)|^ph(\z)^{p}\left(\int_{\D}|B^\om(z,\z)|\frac{v(z)}{\widehat{v}(z)^{1/p'}}\,dA(z)\right)dA(\z).
    \end{split}
    \end{equation}
Another application of Theorem~\ref{th:kernelstimate}(i) and an integration by parts give
    \begin{equation*}
    \begin{split}
    \int_{\D}|B^\om(z,\z)|\frac{v(z)}{\widehat{v}(z)^{1/p'}}\,dA(z)
    &\lesssim \int_0^{1}\frac{v(s)}{\widehat{v}^{1/p'}(s)}\left(\int_0^{s|\z|}\frac{dt}{\widehat{\om}(t)(1-t)}\right)\,ds
    \\ & =  \int_0^{|\z|}\left(\int_{t/|\z|}^1\frac{v(s)}{\widehat{v}^{1/p'}(s)}\,ds\right)\frac{dt}{\widehat{\om}(t)(1-t)}
    \\ & \le \int_0^{|\z|}\widehat{\left(\frac{v}{\widehat{v}^{1/p'}}\right)}(t)\frac{dt}{\widehat{\om}(t)(1-t)}
    \asymp \int_0^{|\z|}\frac{\widehat{v}^{1/p}(t)}{\widehat{\om}(t)}\,\frac{dt}{(1-t)},
    \end{split}
    \end{equation*}
and so
    \begin{equation*}
    \begin{split}
    &h(\z)^{p}\left(\int_{\D}|B^\om(z,\z)|\frac{v(z)}{\widehat{v}(z)^{1/p'}}\,dA(z)\right)\\
    &\lesssim
    v(\z)\left(\int_{|\z|}^1\left(\frac{\om(s)}{v(s)}\right)^{p'}\,v(s)ds\right)^{\frac{1}{p'}}
    \int_0^{|\z|}\frac{\widehat{v}^{1/p}(t)}{\widehat{\om}(t)}\,\frac{dt}{(1-t)}.
    \end{split}
    \end{equation*}
By Bernoulli-l'H\^{o}pital theorem and the hypotheses (c) and $\om,v\in\R$, we deduce
    \begin{equation*}
    \begin{split}
    &\limsup_{r\to 1^-}
    \frac{\int_0^{r}\frac{\widehat{v}^{1/p}(t)}{\widehat{\om}(t)}\,\frac{dt}{(1-t)}}
    {\left(\int_{r}^1\left(\frac{\om(s)}{v(s)}\right)^{p'}\,v(s)ds\right)^{-\frac{1}{p'}}}
    \\ & \lesssim
    \limsup_{r\to 1^-} \left(\int_{r}^1\left(\frac{\om(s)}{v(s)}\right)^{p'}\,v(s)ds\right)^{1+\frac{1}{p'}}\left(\frac{\widehat{v}(r)^{1/p}(1-r)^\frac1{p'}}{\widehat{\om}(r)}\right)^{p'}
    \frac{\widehat{v}^{1/p}(r)}{\widehat{\om}(r)}\frac{1}{1-r}
    \\ & \lesssim
    \limsup_{r\to 1^-}
    \left(\frac{\widehat{\om}(r)}{\widehat{v}^{1/p}(r)}\right)^{p'+1}
    \left(\frac{\widehat{v}(r)^{1/p}}{\widehat{\om}(r)}\right)^{p'}
    \frac{\widehat{v}^{1/p}(r)}{\widehat{\om}(r)}
    =1,
    \end{split}
    \end{equation*}
and consequently,
    $$
    h(\z)^{p}\left(\int_{\D}|B^\om(z,\z)|\frac{v(z)}{\widehat{v}(z)^{1/p'}}\,dA(z)\right)
    \lesssim v(\z).
    $$
This and \eqref{eq:mu4} give $\|P^+_\om(f)\|_{L^p_v}\lesssim\|f\|_{L^p_v}$, and thus we have shown that (a), (b) and (c) are equivalent.

The condition (c) is equivalent to saying that $\left(\frac{\om}{v^{1/p}}\right)^{p'}\in\R$ because $\om,v\in\R$ by the hypothesis. An application of Lemma~\ref{Lemma:Muckenhoupt} below, with $\alpha=p$, now implies that (c), (d), (e) and (f) are equivalent. Further, (f) together with the hypotheses $\om,v\in\R$ shows that $\frac{\om(r)}{(1-r)^{1/p}v(r)^{1/p}}$ is a regular weight, and another application of Lemma~\ref{Lemma:Muckenhoupt}, with $\alpha=2$, gives (f)$\Leftrightarrow$(g).
\end{proof}

For $\om:[0,1)\to(0,\infty)$, define
    $$
    \widetilde{\psi}_\om(r)=\frac{1}{\om(r)}\int_0^r\om(s)\,ds,\quad r\in[0,1),
    $$
and recall that, for each weight $\om$,
    $$
    \psi_\om(r)=\frac1{\om(r)}\int_r^1\om(s)\,ds.
    $$
Several useful characterizations of regular weights are gathered to the
following lemma, the proof of which is standard and therefore omitted.

\begin{lemma}\label{Lemma:Muckenhoupt}
Let $\om$ be a radial weight and $1<\a<\infty$. Denote $\om_1(r)=\om(r)^{1-\a}(1-r)^{-\a}$ and $\om_2(r)=(\om(r)(1-r))^{-\frac1\a}\om(r)$. Then the following assertions are equivalent:
\begin{itemize}
\item[\rm(i)] $\om\in\R$;
\item[\rm(ii)] $\displaystyle
    \frac{\widetilde{\psi}_{\om_1}(r)}{1-r}\asymp1,\quad r\to1^-; 
    $
\item[\rm(iii)] $\om$ satisfies \eqref{eq:r2} and
    \begin{equation}\label{Eq:Muckenhoupt}
    \sup_{0<r<1}\left(\frac{\widetilde{\psi}_{\om_1}(r)}{1-r}\right)\left(\frac{\psi_\om(r)}{1-r}\right)^{\a-1}<\infty;
    \end{equation}
\item[\rm(iv)] $\om_2\in\R$.
\end{itemize}
\end{lemma}

We now deal with the case $p=1$.

\medskip

\noindent\emph{Proof of Theorem~\ref{theorem:projections5}.} By \eqref{eq:adjoint}, $P_\om:L^1_v\to L^1_v$ is bounded if and only if
\begin{equation}\label{eq:case11}
\sup_{\z\in\D}\left|\frac{\om(\z)}{v(\z)}\int_{\D}g(z)B^\om(\z,z)v(z)\,dA(z)\right|\lesssim\|g\|_{L^\infty},\quad g\in L^\infty.
\end{equation}
For each $a\in \D$, define
    \begin{equation*}
    g_a(z)=\begin{cases}
    &\frac{|B^\om(a,z)|}{B^\om(a,z)},\quad \text{if $B^\om(a,z)\neq 0$}\\
    & 0, \quad\quad \text{if $B^\om(a,z)=0$}.
    \end{cases}
    \end{equation*}
By using this family as test functions, \eqref{eq:case11} shows that (a) is equivalent to
    \begin{equation*}\label{eq:case12}
    \sup_{\z\in\D}\frac{\om(\z)}{v(\z)}\int_{\D}\left|B^\om(\z,z)\right|v(z)\,dA(z)<\infty,
    \end{equation*}
which is in turn equivalent to (c) by Theorem~\ref{th:kernelstimate}.

A similar argument shows that the condition (c) characterizes the bounded linear operators
    $$
    \widetilde{P^+_\om}(f)(z)=\int_{\D}f(\z)|B^\om(z,\z)|\om(\z)\,dA(\z)
    $$
on $L^1_v$, and since clearly $\widetilde{P^+_\om}$ is bounded on $L^1_v$ if and only if (b) is satisfied, we have shown that (a), (b) and (c) are equivalent.

To complete the proof it suffices to notice that (d) says that $\frac{\om(r)}{(1-r)v(r)}$ is regular, which is equivalent to (c) by Lemma~\ref{Lemma:Muckenhoupt} with $\a=2$.
\hfill$\Box$

\subsection{Self-improving conditions}

The following lemma together with Theorem~\ref{theorem:projections3} shows that the equivalent conditions appearing in Theorem~\ref{theorem:projections3} are self-improving in the sense that if one of them is satisfied for some $p>1$, then all of them are satisfied for $p-\d$ in place of $p$ if $\d>0$ is sufficiently small.

\begin{lemma}\label{Lemma:self-improving}
Let $0<p<\infty$ and $\om,v\in\R$ such that
    \begin{equation}\label{eq:(d)}
   \sup_{0<r<1} \frac{\widehat{\om}(r)^p}{\widehat{v}(r)}\int_0^r\frac{\widehat{v}(s)}{\widehat{\om}(s)^p(1-s)}\,ds<\infty.
    \end{equation}
Define
   \begin{equation*}
    m=\sup\left\{\d\ge0: \sup_{0<r<1}\frac{\widehat{\om}(r)^{p-\d}}{\widehat{v}(r)}\int_0^r\frac{\widehat{v}(s)}{\widehat{\om}(s)^{p-\d}(1-s)}\,ds<\infty
    \right\}
    \end{equation*}
and
    $$
    M=\inf\left\{\d\in\mathbb{R}\,:\int_0^1\frac{\widehat{v}(s)}{\widehat{\om}(s)^{p-\d}(1-s)}\,ds<\infty 
    \right\}.
    $$
Then $0<m\le M<p$. Moreover, if $\kappa_\om=\lim_{r\to1^-}\frac{\psi_\om(r)}{1-r}$ and $\kappa_v=\lim_{r\to1^-}\frac{\psi_v(r)}{1-r}$ exist, then $m=M=p-\frac{\kappa_\om}{\kappa_v}$. In particular, $\frac{\kappa_\om}{\kappa_v}<p$.
\end{lemma}

\begin{proof}
Without loss of generality we may assume that $\om$ and $v$ are continuous. If the integral condition in the definition of $m$ is satisfied for some $\d_0>0$, then it is satisfied for all $\d\le\d_0$; and similarly, if it fails for $\d_0>0$, then it fails for all $\d\ge\d_0$. Further, the condition obviously fails for $\d=p$, and also for $\d=\d(v,\om)<p$ sufficiently large because then $\widehat{v}(r)\widehat{\om}(r)^{\d-p}\to0$, $r\to1^-$, by \cite[p.~10]{PelRat} since $\om,v\in\R$ by the hypothesis. This implies $m<p$.
Furthermore, an integration by parts and the hypothesis \eqref{eq:(d)} show that
    $$
    \int_0^r\frac{\widehat{v}(s)}{\widehat{\om}(s)^{p-\d}(1-s)}\,ds\lesssim\frac{\widehat{v}(r)}{\widehat{\om}(r)^{p-\d}}
    +\d\int_0^r\frac{\widehat{v}(s)}{\widehat{\om}(s)^{p-\d}(1-s)}\,ds,\quad \d>0,
    $$
and hence
    \begin{equation*}
    \frac{\widehat{\om}(r)^{p-\d}}{\widehat{v}(r)}\int_0^r\frac{\widehat{v}(s)}{\widehat{\om}(s)^{p-\d}(1-s)}\,ds\lesssim1,\quad r\to1^-,
    \end{equation*}
for all $\d>0$ sufficiently small. Thus $m\in(0,p)$.
The integral in the definition of $M$ is bounded for $\d=p$ because $v\in\R$. If $\int_0^1\frac{\widehat{v}(s)\,ds}{\widehat{\om}(s)^{p-\d_0}(1-s)}$ converges, so does the same integral with $\d\ge\d_0$ in place of~$\d_0$.
Moreover, since $\om,v\in\R$ by the hypothesis, by \cite[p. $10$ (ii)]{PelRat},  $v/\widehat{\om}^\a$ is a (regular) weight for $\a>0$ sufficiently small, and thus $M<p$. To see that $m\le M$, assume on the contrary that there exists $\d>0$ such that
    $$
    \frac{\widehat{\om}(r)^{p-\d}}{\widehat{v}(r)}\int_0^r\frac{\widehat{v}(s)}{\widehat{\om}(s)^{p-\d}(1-s)}\,ds\lesssim1,\quad r\to1^-,\quad \textrm{and}\quad \int_0^1\frac{\widehat{v}(s)}{\widehat{\om}(s)^{p-\d}(1-s)}\,ds<\infty.
    $$
Then we deduce $\widehat{\om}(r)^{p-\d}/\widehat{v}(r)\lesssim1$, as $r\to1^-$, implying that $v/\widehat{\om}^{p-\d}$ is not a weight. This is obviously a contradiction, and thus $0<m\le M<p$.

Assume now that $\kappa_\om=\lim_{r\to1^-}\frac{\psi_\om(r)}{1-r}$ and $\kappa_v=\lim_{r\to1^-}\frac{\psi_v(r)}{1-r}$ exist. Then a direct calculation shows that for a given $\e>0$,
    \begin{equation*}
    \begin{split}
    (1-r)^{\frac{1}{\kappa_\om}+\e}\lesssim\widehat{\om}(r)\lesssim(1-r)^{\frac{1}{\kappa_\om}-\e}\quad \textrm{and} \quad
    (1-r)^{\frac{1}{\kappa_v}+\e}\lesssim\widehat{v}(r)\lesssim(1-r)^{\frac{1}{\kappa_v}-\e}
    \end{split}
    \end{equation*}
for all $r\in[0,1)$, and further, $\widehat{v}(r)(1-r)^{-\frac1{\kappa_v}-\e}$ and $\widehat{\om}(r)^{-1}(1-r)^{\frac1{\kappa_\om}-\e}$ are essentially increasing on $[0,1)$, see \cite[(ii) p.~10]{PelRat} for details. To prove $m=M$, let $K<p-\frac{\kappa_\om}{\kappa_v}$ be fixed. Then, for $\e>0$ sufficiently small,
    \begin{equation}
    \begin{split}\label{pkwkv}
    \int_0^r\frac{\widehat{v}(s)}{\widehat{\om}(s)^{p-K}(1-s)}\,ds
    &\asymp
    \int_0^r\left(\frac{\widehat{v}(s)}{(1-s)^{\frac{1}{\kappa_v}+\e}}\right)
    \left(\frac{(1-s)^{\frac1{\kappa_\om}-\e}}{\widehat{\om}(s)}\right)^{p-K}\frac{ds}{(1-s)^{\frac{p-K}{\kappa_\om}-\frac1{\kappa_v}+1-\e(p-K+1)}}\\
    &\lesssim\frac{\widehat{v}(r)}{\widehat{\om}(r)^{p-K}}(1-r)^{\frac{p-K}{\kappa_\om}-\frac1{\kappa_v}-\e(p-K+1)}
    \int_0^r\frac{ds}{(1-s)^{\frac{p-K}{\kappa_\om}-\frac1{\kappa_v}+1-\e(p-K+1)}}\\
    &\lesssim \frac{\widehat{v}(r)}{\widehat{\om}(r)^{p-K}},\quad r\to1^-,
    \end{split}
    \end{equation}
and hence, $m\ge p-\frac{\kappa_\om}{\kappa_v}$. Similarly, for a fixed $K\in(p-\frac{\kappa_\om}{\kappa_v},p)$ and $\e>0$ sufficiently small,
    \begin{equation}
    \begin{split}\label{pkwkv2}
    \int_0^1\frac{\widehat{v}(s)}{\widehat{\om}(s)^{p-K}(1-s)}\,ds
    &\asymp
    \int_0^1\left(\frac{\widehat{v}(s)}{(1-s)^{\frac{1}{\kappa_v}-\e}}\right)
    \left(\frac{(1-s)^{\frac1{\kappa_\om}+\e}}{\widehat{\om}(s)}\right)^{p-K}\frac{ds}{(1-s)^{\frac{p-K}{\kappa_\om}-\frac1{\kappa_v}+1+\e(p-K+1)}}\\
    &\lesssim
    \int_0^1\frac{ds}{(1-s)^{\frac{p-K}{\kappa_\om}-\frac1{\kappa_v}+1+\e(p-K+1)}}
    \lesssim1,
    \end{split}
    \end{equation}
and consequently, $M\le p-\frac{\kappa_\om}{\kappa_v}\le m$. It follows that $m=M=p-\frac{\kappa_\om}{\kappa_v}$ as claimed. Moreover, since $m>0$, we deduce $\frac{\kappa_\om}{\kappa_v}<p$.
\end{proof}

\begin{theorem}\label{theorem:projections3special}
Let $1\le p<\infty$ and $\om,v\in\R$ such that $\kappa_\om$ and $\kappa_v$ exist. Then the following conditions are equivalent:
\begin{itemize}
\item[\rm(a)] $P^+_\om:L^p_v\to L^p_v$ is bounded;
\item[\rm(b)] $P_\om:L^p_v\to L^p_v$ is bounded;
\item[\rm(c)] $\displaystyle \sup_{0<r<1}\frac{\om(r)^p(1-r)^{p-1}}{v(r)}\int_0^r\frac{v(s)}{\om(s)^p(1-s)^p}\,ds<\infty$;
\item[\rm(d)] $\displaystyle\frac{\kappa_\om}{\kappa_v}<p$.
\end{itemize}
\end{theorem}

\begin{proof}
The conditions (a)-(c) are equivalent by Theorems~\ref{theorem:projections5} and~\ref{theorem:projections3}. Further, (c)$\Rightarrow$(d) follows by Lemma~\ref{Lemma:self-improving}.

Finally, assume that (d) is satisfied, that is, $\frac{\kappa_\om}{\kappa_v}<p$, and let us prove (c).
The reasoning in \eqref{pkwkv} yields $m\ge p-\frac{\kappa_\om}{\kappa_v}>0$,
so that
    \begin{equation}\label{xxxx}
    \frac{\widehat{\om}(r)^{p-\e}}{\widehat{v}(r)}\int_0^r\frac{\widehat{v}(s)}{\widehat{\om}(s)^{p-\e}(1-s)}\,ds\lesssim1,\quad r\to1^-,
    \end{equation}
for all
$\e<p-\frac{\kappa_\om}{\kappa_v}$, which in particular implies (c). This finishes the proof.
\end{proof}

\subsection{One weight inequality}

In this section we prove Theorem~\ref{theorem:projections} and deduce Theorem~\ref{th:L1pI} from a result of Shields and Williams.

\medskip

\noindent\emph{Proof of Theorem~\ref{theorem:projections}.}
(i). This follows from Theorem~\ref{theorem:projections3}.

(ii). Let $\phi\in L^\infty(\D)$. Then
    $$
    (P_\om(\phi))'(z)=\int_{\D} \phi(\z)\frac{\partial B^\om(z,\z)}{\partial z}\,\om(\z)dA(\z),
    $$
and hence Theorem~\ref{th:kernelstimate}(ii) gives
    $$
    |(P_\om(\phi))'(z)|\le \|\phi\|_{L^\infty(\D)}\int_{\D} \left|\frac{\partial B^\om(z,\z)}{\partial z}\right|\,\om(\z)dA(\z)\asymp \frac{\|\phi\|_{L^\infty(\D)}}{1-|z|}.
    $$
Since $|P_\om(\phi)(0)|\le C\|\phi\|_{L^\infty(\D)}$ for some constant $C=C(\om)>0$, it follows that $P_\om(\phi)\in \mathcal{B}$ and $\|P_\om(\phi)\|_{\mathcal{B}}\le C\|\phi\|_{L^\infty(\D)}$.

(iii). Let first $p>1$. We assume that $P^+_\om:L^p_\om\to L^p_\om$ is bounded and aim for a contradiction. Write  $K(r)=\int_0^{r}\frac{dt}{\widehat{\om}(t)(1-t)}$ for short, and let $\phi$ be a radial function. Then Theorem~\ref{th:kernelstimate}(i)
 together with Lemma~\ref{Lemma:replacement-Lemmas-Memoirs} show that
    \begin{equation*}
    \begin{split}
    P_\om^+(\phi)(z) &\asymp \int_{0}^1 K(|z|s)\phi(s)\om(s)\,ds\ge
    K(|z|^2)\int_{|z|}^1 \phi(s)\om(s)\,ds\\
    &\asymp K(|z|)\int_{|z|}^1 \phi(s)\om(s) \,ds,\quad |z|\ge\frac12.
    \end{split}
    \end{equation*}
Therefore
    $$
    \|P^+_\om(\phi)\|^p_{L^p_\om}\gtrsim  \int_0^1 \left(K(r)\int_r^1 \phi(s)\om(s)\,ds\right)^p\om(r)\,dr,
    $$
and since we assumed that $P^+_\om:L^p_\om\to L^p_\om$ is bounded, we deduce
    \begin{equation}\label{muck1xxxn}
    \int_0^1 \left(K(r)\int_r^1 \phi(s)\om(s)\,ds\right)^p\om(r)\,dr \lesssim\|\phi\|^p_{L^p_\om},\quad \phi\in L^p_\om.
    \end{equation}
By choosing $\phi_t=\chi_{[t,1)}$, we obtain
\begin{equation*}\begin{split}
\widehat{\om}(t) &\gtrsim \int_0^1 \left(K(r)\int_{\max\{r,t\}}^1\om(s)\,ds\right)^p\om(r)\,dr
\ge \widehat{\om}(t)^p\int_0^t K(r)^p\om(r)\,dr,
\end{split}\end{equation*}
that is,
    \begin{equation}\label{muck7}
    \sup_{0<r<1}\left(\int_0^rK^p(s)\om(s)\,ds\right)
    \widehat{\om}(r)^{\frac{p}{p'}}<\infty.
    \end{equation}
Since $K(r)\gtrsim\widehat{\om}(r)^{-1}$,
we deduce
    $$
    \int_0^rK^p(s)\om(s)\,ds\gtrsim \int_0^r\frac{\om(s)}{\widehat{\om}(s)^p}\,ds\to \infty,\quad r\to 1^-.
    $$
Two applications of the Bernoulli-l'H\^{o}pital
theorem now give
    \begin{equation}\label{pillu}
    \begin{split}
    \liminf_{r\to 1^-}\frac{\int_0^rK^p(s)\om(s)\,ds}{\widehat{\om}(r)^{-\frac{p}{p'}}}
    &\ge \frac{1}{p-1}\liminf_{r\to 1^-} \frac{K^p(r)}{\widehat{\om}(r)^{-p}}\\
    &= \frac{1}{p-1}\left(\liminf_{r\to 1^-} \frac{K(r)}{\widehat{\om}(r)^{-1}}\right)^p\\
    &\ge \frac{1}{p-1}\liminf_{r\to 1^-} \left(\frac{\psi_{\om}(r)}{1-r}\right)^p=\infty.
    \end{split}
    \end{equation}
Therefore \eqref{muck7} is false and consequently, $P^+_\om:L^p_\om\to L^p_\om$ is not bounded.\hfill$\Box$

\medskip

Before proving Theorem~\ref{th:L1pI}, it is worth noticing that the hypotheses of the theorem are satisfied for example, if $\om(r)=(1-r)^{-1}\left(\log\frac{e}{1-r}\right)^{-\a}$, $\a>1$, or more generally,
    \begin{equation*}
    \om(r)=\left((1-r)\prod_{n=1}^{N}\log_n\frac{\exp_{n}0}{1-r}\left(\log_{N+1}\frac{\exp_{N+1}0}{1-r}\right)^\a\right)^{-1},
    \end{equation*}
where $1<\a<\infty$ and $N\in\N$. Here, as usual,
$\log_nx=\log(\log_{n-1}x)$, $\log_1x=\log x$,
$\exp_n x=\exp(\exp_{n-1}x)$ and $\exp_1x=e^x$.

\medskip

\noindent\emph{Proof of Theorem~\ref{th:L1pI}.} By \cite[Lemmas~1.1 and~1.3]{PelRat}, the function
$\Psi$ satisfies
    $$
    \Psi(x)\asymp\frac{1}{\int_0^1 s^x\om(s)\,ds},\quad x\in[0,1).
    $$
Since $\om\in\I$ by the hypothesis, for a given $a>0$, the function $h(r)=\frac{\widehat\om(r)}{(1-r)^a}$ is increasing on $[\rho,1)$ for some $\rho=\rho(a)\in(0,1)$. So $\Psi$ satisfies condition (U) in \cite[p. 5]{ShiWiMich82}.
Therefore we may apply \cite[Lemma~2 and Theorem~3]{ShiWiMich82} with $d\eta(r)=r\om(r)\,dr$ to deduce
that if there were a bounded projection from $L^1_\om$ to $A^1_\om$, then the function
    $$
    x\mapsto\widehat\om\left(1-\frac{1}{x+1}\right)\int_{1/2}^x\frac{dt}{\widehat\om\left(1-\frac{1}{t+1}\right)t}
    $$
would be bounded. But this is impossible as is seen by the change of variable $1-\frac{1}{x+1}=r$ and an application of the Bernoulli-l'H\^{o}pital
theorem similar to that in the last step in \eqref{pillu}.
Thus there are no bounded projections from $L^1_\om$ to $A^1_\om$. \hfill$\Box$

\section{Duality}

Recall that $V_{p'}=V_{p'}(\om,v)=\left(\frac{\om}{v}\right)^{p'}v$.

\begin{proposition}\label{pr:duality}
Let $1<p<\infty$ and $\om\in\R$, and let $v$ be a radial weight. Then the following assertions are equivalent:
\begin{itemize}
\item[\rm(a)] $v\in\R$ and $P_\om: L^p_v\to L^p_v$ is bounded;
\item[\rm(b)] $V_{p'}\in\R$  and $P_\om: L^{p'}_{V_{p'}}\to L^{p'}_{V_{p'}}$ is bounded.
    \end{itemize}
\end{proposition}

\begin{proof} (a)$\Rightarrow$(b). If $P_\om: L^p_v\to L^p_v$ is bounded, then $V_{p'}\in\R$ by Theorem~\ref{theorem:projections3}. Moreover, $v=\left(\frac{\om}{V_{p'}}\right)^{p}V_{p'}\in\R$, so $P_\om:L^{p'}_{V_{p'}}\to L^{p'}_{V_{p'}}$ is bounded by Theorem~\ref{theorem:projections3}.

A reasoning analogous to that above gives (b)$\Rightarrow$(a).
\end{proof}

\noindent\emph{Proof of Theorem~\ref{th:duality3}.} (a)$\Rightarrow$(b). Denote $\Lambda_g(f)=\langle f,g\rangle_{A^2_\om}$ for all $f\in A^{p'}_{V_{p'}}$ and $g\in A^p_v$. Then
    \begin{equation*}
    |\Lambda_g(f)|=\left|\langle f,g\rangle_{A^2_\om} \right|\le \int_{\D}|f(z)| |g(z)| \frac{\om(z)}{v(z)}\,v(z)dA(z)\le \|f\|_{A^{p'}_{V_{p'}}}\|g\|_{A^p_v},
    \end{equation*}
so $\Lambda_g\in(A^{p'}_{V_{p'}})^\star$ and $\|\Lambda_g\|\le\|g\|_{A^p_v}$.

Conversely, if $T\in(A^{p'}_{V_{p'}})^\star$, then $T$
can be extended to a bounded linear functional $\widetilde{T}$ on $L^{p'}_{V_{p'}}$ with $\|T\|=\|\widetilde T\|$ by the Hahn-Banach theorem. Now that $(L^{p'}_{V_{p'}})^\star$ is isometrically isomorfic to $L^{p}_{v}$, and hence there exists $h\in L^{p}_v$ such that $T(f)=\langle f,h\rangle_{L^2_\om}$ for all $f\in L^{p'}_{V_{p'}}$, and $\|\widetilde T\|=\|h\|_{L^{p}_v}$.
Since $P^+_\om$ is bounded on $L^{p'}_{V_{p'}}$ by the hypothesis (a), Theorem~\ref{theorem:projections2} and Proposition~\ref{pr:duality}, Fubini's theorem yields
    \begin{equation*}
    T(f)=T(P_\om(f))= \langle f, P_\om(h)\rangle_{A^2_\om}=\Lambda_g(f),\quad f\in A^{p'}_{V_{p'}},
    \end{equation*}
where $g=P_\om(h)\in A^{p}_v$ satisfies $\|g\|_{A^p_v}\le\|P_\om\|\|h\|_{L^{p}_{v}}=\|P_\om\|\|T\|$ by the hypothesis (a). Therefore we have proved (b).

(b)$\Rightarrow$(a). Assume $\left(A^{p'}_{V_{p'}}\right)^\star \simeq A^p_v$ up to an equivalence of norms under the pairing \eqref{pairingom}. Let $h\in L^p_v$, and consider the bounded linear functional $T_h(f)=\langle f,h\rangle_{L^2_\om}$ on $L^{p'}_{V_{p'}}$, with $\|T_h\|=\|h\|_{L^p_v}$.
By Fubini's theorem $T_h(f)=\langle f, P_\om(h)\rangle_{A^2_\om}$ for every polynomial~$f$.
Further, by the hypothesis, there exists $g\in A^{p}_{v}$ such that $T_h(f)=\langle f, P_\om(h)\rangle_{A^2_\om}=\langle f,g\rangle_{A^2_\om}$ for all $f\in A^{p'}_{V_{p'}}$ and $\|T_h\|\asymp\|g\|_{A^{p}_v}$.
So testing on the monomials $\{z^n\}_{n\in\N\cup\{0\}}$, we deduce $P_\om(h)=g$, and hence $\|P_\om(h)\|_{A^{p}_{v}}\asymp\|h\|_{L^{p}_{v}}$.
Therefore $P_\om: L^{p}_v\to L^{p}_{v}$ is bounded.
\hfill$\Box$

\smallskip

\noindent\emph{Proof of Corollary~\ref{th:duality2}.} (i) follows by Theorem~\ref{th:duality3}.

(ii). We begin with showing that each $g\in\B$ induces a bounded linear functional on~$A^1_\om$.
By the polarization of the identity
\eqref{LP1}, Lemma~\ref{Lemma:replacement-Lemmas-Memoirs} and \eqref{LPformula}, we deduce
    \begin{equation*}
    \begin{split}\label{eq:dual1}
    |\langle f,g\rangle_{A^2_\om}|& \lesssim |f(0)||g(0)|+|\langle f',g'\rangle_{A^2_{\om^\star}}|\\
    &\lesssim \|f\|_{A^1_\om}\|g\|_{\B}+\int_{\D}|f'(z)||g'(z)|(1-|z|)^2\om(z)\,dA(z)
   \lesssim \|g\|_{\B}\|f\|_{A^1_\om}.
    \end{split}
    \end{equation*}

Assume next that $L$ is a bounded linear functional on $A^{1}_\om$. By the Hahn-Banach theorem $L$ can be extended
to a bounded linear functional $\widetilde L$ on $L^{1}_\om$ with $\|L\|=\|\widetilde L\|$.
So,
there exists a unique function $h\in L^{\infty}(\D)$ such that
    $$
    \widetilde{L}f=\int_\D  f(z)\overline{h(z)}\om(z)\,dA(z), \quad f\in L^{1}_\om,
    $$
and $\|\widetilde{L}\|=\|h\|_{L^{\infty}(\D)}$. By using the restriction of this identity to functions in $A^1_\om$, and Fubini's theorem we deduce
    \begin{equation*}
    \begin{split}
    L f&=\widetilde{L}f= \lim_{r\to 1^-}\int_\D  f(rz) \overline{h(z)}\om(z)\, dA(z)\\
    &=\lim_{r\to 1^-} \int_\D \left(\int_\D f(r\z) B^\om(z,\z)\om(\z)\, dA(\z) \right)\overline{h(z)}\om(z)\, dA(z)\\
    &=\lim_{r\to 1^-} \int_\D  f(r\z)   \overline{P_\om(h) (\z)}\om(\z)\,dA(\z).
    \end{split}
    \end{equation*}
The first part of the proof implies that this last limit equals to $\langle f,P_\om(h)\rangle_{A^2_\om}$, because $P_\om:L^\infty(\D)\to\B$ is bounded by Theorem~\ref{theorem:projections}(ii). Thus $\|P_\om (h) \|_{\B}\lesssim \|h\|_{L^{\infty}(\D)}=\|\widetilde{L}\|=\|L\|$, and the assertion is proved.\hfill$\Box$

\end{document}